\UseRawInputEncoding

\documentclass[a4paper]{amsart}
\usepackage{latexsym,bm,stmaryrd}
\usepackage{amsmath,amsthm,amsfonts,amssymb,mathrsfs,pb-diagram}
\allowdisplaybreaks[4]
\usepackage{microtype}
\usepackage{mathtools}
\usepackage{mathbbol,wasysym}
\usepackage[misc]{ifsym}

\let\<=\langle
\let\>=\rangle
\def\({\big(}
\def\){\big)}

\def\Z{\mathbb{Z}}
\def\Q{\mathbb{Q}}

\def\N{\mathbb{N}}

\def\mR{\mathcal{R}}

\def\O{\mathcal{O}}

\def\lam{\lambda}
\def\Lam{\Lambda}
\def\Sym{\mathfrak{S}}

\def\bc{\mathbf{c}}
\newcommand\RR{\mathscr{R}}
\newcommand\R[1][n]{\RR_{#1}^{\Lambda}}

\def\bn[#1,#2]{\begin{bmatrix}#1\\#2\end{bmatrix}}

\def\fg{\mathfrak{g}}

\DeclareMathOperator\lmod{\!-mod}
\DeclareMathOperator\lgmod{\!-gmod}
\DeclareMathOperator\Hom{Hom}
\DeclareMathOperator\End{End}

\DeclareMathOperator\soc{soc}
\DeclareMathOperator\im{Im}
\DeclareMathOperator\head{hd}
\DeclareMathOperator\rad{rad}
\DeclareMathOperator\cha{char}

\DeclareMathOperator\id{id}
\DeclareMathOperator\Tr{Tr}

\DeclareMathOperator\diag{diag}

\DeclareMathOperator\NH{NH}
\DeclareMathOperator\Irr{Irr}

\title[Cocenter of the cyclotomic quiver Hecke algebras]
{On the maximal and minimal degree components of the cocenter of the cyclotomic KLR algebras}
\subjclass[2010]{20C08, 16G99, 06B15}
\keywords{Cyclotomic quiver Hecke algebras, categorification}
\author{Jun Hu}\address{Key Laboratory of Algebraic Lie Theory and Analysis of Ministry of Education\\
School of Mathematics and Statistics\\
  Beijing Institute of Technology\\
  Beijing, 100081, P.R. China}
\email{junhu404@bit.edu.cn}

\author[Lei Shi]{Lei Shi\textsuperscript{\Letter}}\thanks{\Letter Lei Shi \qquad Email: 3120195738@bit.edu.cn}
\address{School of Mathematics and Statistics\\
  Beijing Institute of Technology\\
  Beijing, 100081, P.R. China}
\email{3120195738@bit.edu.cn}

\newtheorem*{conja}{Conjecture A}

\numberwithin{equation}{section}
\newtheorem{prop}[equation]{Proposition}
\newtheorem{thm}[equation]{Theorem}

\newtheorem{cor}[equation]{Corollary}

\newtheorem{lem}[equation]{Lemma}

\theoremstyle{definition}
\newtheorem{dfn}[equation]{Definition}
\newtheorem{conv}[equation]{Convention}
\theoremstyle{remark}
\newtheorem{rem}[equation]{Remark}

\begin{document}

\begin{abstract} Let $\R[\alpha]$ be the cyclotomic KLR algebra associated to a symmetrizable Kac-Moody Lie algebra $\mathfrak{g}$ and polynomials $\{Q_{ij}(u,v)\}_{i,j\in I}$. Shan, Varagnolo and Vasserot show that, when the ground field $K$ has characteristic $0$, the degree $d$ component of the cocenter $\Tr(\R[\alpha])$ is nonzero only if $0\leq d\leq d_{\Lam,\alpha}$. In this paper we show that this holds true for arbitrary ground field $K$, arbitrary $\fg$ and arbitrary polynomials $\{Q_{ij}(u,v)\}_{i,j\in I}$. We generalize our earlier results \cite[Theorem 1.3]{HS3} on the $K$-linear generators of $\Tr(\R[\alpha]), \Tr(\R[\alpha])_0, \Tr(\R[\alpha])_{d_{\Lam,\alpha}}$ to arbitrary ground field $K$. Moreover, we show that the dimension of the degree $0$ component $\Tr(\R[\alpha])_0$ is always equal to $\dim V(\Lam)_{\Lam-\alpha}$, where $V(\Lam)$ is the integrable highest weight $U(\mathfrak{g})$-module with highest weight $\Lam$, and we obtain a basis for $\Tr(\R[\alpha])_0$.
\end{abstract}

\maketitle
\setcounter{tocdepth}{1}

\section{Introduction}

To each symmetrizable Kac-Moody Lie algebra $\fg$, an element $\alpha$ in the root lattice $Q_n^+$ with height $n$, and some polynomials $\{Q_{ij}(u,v)\}_{i,j\in I}$, Khovanov, Lauda (\cite{KL1}, \cite{KL2}) and Rouquier (\cite{Rou2}, \cite{Rou1}) associated a remarkable infinite family of $\Z$-graded $K$-algebras $\RR_\alpha$ which are nowadays called quiver Hecke algebras or KLR algebras. These algebras can be used to give categorification of the negative half of the quantum groups $U_q(\fg)$, see \cite{KL1} and \cite{Rou2}. The KLR algebras have some remarkable finite dimensional quotients $\R[\alpha]$ called the cyclotomic quiver Hecke algebras or the cyclotomic KLR algebras, where $\Lam$ is an integral dominant weight of $\fg$. In \cite{KK} and \cite{VV}, it was proved that the cyclotomic KLR algebras $\R[\alpha]$ can be used to give categorification of the integrable highest weight module $V_q(\Lam)$ over the quantum groups $U_q(\fg)$. In the past decade, various important applications of the KLR algebras and their cyclotomic quotients have been found, see \cite{Bow}, \cite{DVV}, \cite{Ev}, \cite{EK}, \cite{HM}, \cite{RW} and the references therein.

We are mostly interested in the structure of the cyclotomic KLR algebras $\R[\alpha]$ over arbitrary ground field $K$ and of {\it arbitrary} types (i.e., for arbitrary symmetrizable Kac-Moody Lie algebra $\fg$ and arbitrary polynomials $\{Q_{ij}(u,v)\}_{i,j\in I}$). In \cite{HS}, closed formulae for these cyclotomic KLR algebras are obtained, which are presented in terms of simple functions involves only the dominant weight $\Lam$ and some simple coroots. Shan, Varagnolo and Vasserot \cite[Proposition 3.10]{SVV} have shown that the algebra $\R[\alpha]$ is a $\Z$-graded symmetric algebra which is equipped with a homogeneous symmetrizing form of degree $-d_{\Lam,\alpha}$. The centers and the cocenters of cyclotomic KLR algebras have attracted much attention in recent years, see \cite{BGHL}, \cite{BHLZ}, \cite{BHLW}, \cite{SVV}, \cite{W0}, \cite{HS3}, \cite{HS4}. The $\Z$-graded symmetric structure on $\R[\alpha]$ implies that there is a $K$-linear isomorphism $(\Tr(\R[\alpha]))^{\circledast}\cong Z(\R[\alpha])\<-d_{\Lam,\alpha}\>$. Assuming $\cha K=0$, Shan, Varagnolo and Vasserot (\cite[Theorem 3.31]{SVV}) have also shown that  $\Tr(\R[\alpha])_d\neq 0$ only if $0\leq d\leq d_{\Lam,\alpha}$ by using some loop operators action on the cocenters induced from a categorical $\mathfrak{sl}_2$ representation. In \cite[Theorem 1.3]{HS3}, we reprove this result for the cyclotomic KLR algebras $\R[\alpha]$ of {\it arbitrary} types by some elementary argument and remove the characteristic $0$ assumption for the ``$d\geq 0$ part'' of the statement. The following theorem is the first main result of this paper.

\begin{thm}\label{mainthm1} Let $K$ be an arbitrary field. Let $\R[\alpha]$ be a cyclotomic KLR algebra over $K$ of arbitrary type. Then $\Tr(\R[\alpha])_d\neq 0$ only if $0\leq d\leq d_{\Lam,\alpha}$ and $d\in 2\Z$. Similarly, $Z(\R[\alpha])_d\neq 0$ only if $0\leq d\leq d_{\Lam,\alpha}$ and $d\in 2\Z$.
\end{thm}

One of the major unsolved open problem on the structure of cyclotomic KLR algebras of arbitrary types is the following conjecture:

\begin{conja}\label{conja} Let $K$ be an arbitrary field. Let $\R[\alpha]$ be a cyclotomic KLR algebra over $K$ of arbitrary type. The $K$-algebra $\R[\alpha]$ is indecomposable.
\end{conja}

In view of the isomorphism $(\Tr(\R[\alpha]))^{\circledast}\cong Z(\R[\alpha])\<-d_{\Lam,\alpha}\>$, Conjecture A will follow if one can show that $\dim\Tr(\R[\alpha])_{d_{\Lam,\alpha}}=1$. To the best of knowledge, Conjecture A is currently known to be true if $\fg$ is simply-laced of finite type (\cite[3.4.1]{SVV} or \cite[3.2]{W0}) and $\cha K=0$ or if $\fg$ is of affine type $A$ with some restriction on $\cha K$ (which ensures Brundan-Kleshchev's isomorphism \cite{BK:GradedKL} exists over some field extension of $K$). In all of these cases the equality $\dim\Tr(\R[\alpha])_{d_{\Lam,\alpha}}=1$ holds (see \cite[3.4.1]{SVV}, \cite[3.2]{W0}, \cite[Theorem 1.10]{HS4}). This motivates our study of the maximal degree component $\Tr(\R[\alpha])_{d_{\Lam,\alpha}}$ of the cocenter $\Tr(\R[\alpha])$. In \cite[Theorem 1.3]{HS3}, under the assumption that $\cha K=0$, spanning sets for the maximal degree component $\Tr(\R[\alpha])_{d_{\Lam,\alpha}}$ of the cocenter $\Tr(\R[\alpha])$ as well as for the minimal degree component $\Tr(\R[\alpha])_{0}$ of the cocenter $\Tr(\R[\alpha])$ are given. The following theorem is the second main result of this paper.

\begin{thm}\label{mainthm2}
Let $K$ be an arbitrary field. Let $\R[\alpha]$ be a cyclotomic KLR algebra over $K$ of arbitrary type. We have \begin{enumerate}
\item [1)] $\bigl(\Tr(\RR^\Lam_\alpha)\bigr)_{d_{\Lam,\alpha}}=\text{$K$-{\rm Span}}\Bigl\{Z^\Lam(\nu)+[\R[\alpha],\R[\alpha]] \Bigm|\begin{matrix}\text{where $\nu$ is piecewise}\\ \text{ dominant with respect to $\Lam$}\end{matrix}\Bigr\}$;
\item [2)] $\bigl(\Tr(\RR^\Lam_\alpha)\bigr)_{0}=\text{$K$-{\rm Span}}\Bigl\{e(\nu)^{(-)}+[\R[\alpha],\R[\alpha]] \Bigm|\begin{matrix}\text{where $\nu$ is piecewise}\\ \text{ dominant with respect to $\Lam$}\end{matrix}\Bigr\}$;
\item [3)]  $\Tr(\R[\alpha])=\text{$K$-{\rm Span}}\Bigl\{a+[\R[\alpha],\R[\alpha]] \Bigm|\begin{matrix}a\in \mR^\Lam(\nu),\,\text{where $\nu$ is piecewise}\\ \text{ dominant with respect to $\Lam$}\end{matrix}\Bigr\}$.
\end{enumerate}
where $Z^\Lam(\nu), e(\nu)^{(-)}, \mR^\Lam(\nu)$ are some explicitly defined elements and set given in Definition (\ref{sdfns}).
\end{thm}

In other words, Theorem \ref{mainthm2} generalize \cite[Theorem 1.3]{HS3} to arbitrary ground field. The element $e(\nu)^{(-)}$ in the above theorem is an idempotent which can be regarded as a ``divided power'' version of $e(\nu)$ inside the cocenter $\Tr(\R[\alpha])$.

In \cite[Theorem 3.31]{SVV}, Shan, Varagnolo and Vasserot have shown that, under the assumption that $\cha K=0$ and the polynomials
$\{Q_{ij}(u,v)|i,j\in I\}$ be of some specific form, see \cite[(11), Theorem 3.25]{SVV}, $\dim\Tr(\R[\alpha])_0=\dim V(\Lam)_{\Lam-\alpha}$, where $V(\Lam)$ is the integrable highest weight $U(\mathfrak{g})$-module with highest weight $\Lam$. Let $\R[\alpha]\lmod$ be the category of finite dimensional ungraded $\R[\alpha]$-modules. The following theorem is the third main result of this paper.

\begin{thm}\label{mainthm3} Let $K$ be an arbitrary field. Let $\R[\alpha]$ be a cyclotomic KLR algebra over $K$ of arbitrary type. Then $\dim\Tr(\R[\alpha])_0=\dim V(\Lam)_{\Lam-\alpha}$. Moreover, if $\{f_j|1\leq j\leq m_{\alpha,\Lam}\}$ is a set of primitive idempotents such that $\{\R[\alpha]f_j|1\leq j\leq m_{\alpha,\Lam}\}$ is a complete set of pairwise non-isomorphic indecomposable projective module in $\R[\alpha]\lmod$, then the following set $$
\bigl\{f_j+[\R[\alpha],\R[\alpha]]\bigm|1\leq j\leq m_{\alpha,\Lam}\bigr\}
$$
gives a $K$-basis of the minimal degree component $\Tr(\R[\alpha])_0$ of the cocenter $\Tr(\R[\alpha])$.
\end{thm}

As a byproduct, we also prove the following Morita equivalence results: $$
\R[\alpha] \overset{Morita}{\sim} e_{\alpha,\Lam}\R[\alpha]e_{\alpha,\Lam}\overset{Morita}{\sim} e_{\alpha,\Lam}^{(-)}\R[\alpha]e_{\alpha,\Lam}^{(-)},
$$
where  $$
e_{\alpha,\Lam}:=\sum_{\nu\in\mathscr{P\!D}_\alpha}e(\nu),\,\, e_{\alpha,\Lam}^{(-)}:=\sum_{\nu\in\mathscr{P\!D}_\alpha}e(\nu)^{(-)},
$$
and $\mathscr{P\!D}_\alpha$ is the set of piecewise dominant sequences in $I^\alpha$ with respect to $\Lam$.
\medskip

The paper is organised as follows. In Section two we recall some preliminary definitions and notations for KLR algebras and cyclotomic KLR algebras. In Section three we first do some nilHecke algebra calculation inside the cocenter and then apply the results to the general cyclotomic KLR algebra. The main new idea is to regard the primitive idempotent in the nilHecke algebra as a ``divided power'' version of the identity element inside the cocenter. For the general cyclotomic KLR algebra we shall replace each KLR idempotent $e(\nu)$ with its ``divided power'' version in each ``nilHecke part'' of $\nu$. The three main results of this paper are all proved in Section four. In particular, we generalize \cite[Theorem 3.31]{SVV} and \cite[Theorem 1.3]{HS3} to arbitrary ground field. It would be interesting to know if one can use the second main result Theorem \ref{mainthm2} 1) to show that $\dim\Tr(\R[\alpha])_{d_{\Lam,\alpha}}=1$, and if similar results hold for the cyclotomic KLRW algebras, which was introduced and studied in \cite{W1}, \cite{W2}, \cite{MT1} and \cite{MT2}.
\bigskip\bigskip

\centerline{Acknowledgements}
\bigskip

The research was supported by the National Natural Science Foundation of China (No. 12171029).
\bigskip

\section{Definitions and notations}

In this section we shall give some preliminary definitions and results on the KLR algebras and their cyclotomic quotients.

\begin{dfn} Let $I$ be an index set. An integral matrix $A = (a_{i,j})_{i,j\in I}$ is called a \emph{symmetrizable Cartan matrix} if it satisfies
\begin{enumerate}
  \item $a_{ii} = 2$, $\forall\, i \in I$;
  \item $a_{ij}\in\Z_{\leq 0}$, $\forall\,i\neq j\in I$;
  \item $a_{ij} = 0 \Leftrightarrow \, a_{ji} = 0, \,\, \forall\,i,j \in I$;
  \item there is a diagonal matrix $D = \diag (d_i\in{\Z}_{>0}\mid i\in I)$ such that $DA$ is symmetric.
\end{enumerate}
\end{dfn}

\begin{dfn} A Cartan datum $(A,P,\Pi , P^{\vee} , \Pi^{\vee})$ consists of
\begin{enumerate}
  \item a symmetrizable Cartan matrix $A$;
  \item a free aphelian group $P$ of finite rank, called the \emph{weight lattice};
  \item $\Pi = \{ \alpha_i \in P \mid i\in I \}$, called the set of \emph{simple roots};
  \item $P^{\vee} := \Hom (P,{\Z})$, called the \emph{dual weight lattice} and $\<-,-\>: P^{\vee} \times P \to {\Z}$, the natural pairing;
  \item $\Pi^\vee = \{ h_i \mid i \in I \} \subset P^{\vee}$, called the set of \emph{simple coroots};
\end{enumerate}
satisfying the following properties:
\begin{enumerate}
  \item[(i)] $\langle h_i , \alpha_j \rangle = a_{ij}, \forall\,i,j\in I$,
  \item[(ii)] $\Pi$ is linearly independent,
  \item[(iii)] $\forall\, i \in I$, $\exists\,\Lambda_i \in P$ such that $\langle h_j, \Lambda_i \rangle = \delta_{ij}$ for all $j \in I$.
\end{enumerate}
\end{dfn}

Those $\Lambda_i$ in the above definition are called the \emph{fundamental weights}. We set
$$P^+ := \{ \Lam \in P \mid \< h_i, \Lam \>\in {\Z}_{\geqslant 0},\,\,\forall\,i\in I \},$$
and call it the set of \emph{dominant integral weights}. The free aphelian group $Q := \oplus_{i\in I}{\Z} \alpha_i$ is called the \emph{root lattice}.
Set $Q^+ = \sum_{i\in I} {\Z}_{\geqslant 0}\alpha_i$. For $\alpha = \sum_{i\in I} k_i \alpha_i \in Q^+$, we define the \emph{height} of $\alpha$ to be $|\alpha| = \sum_{i\in I} k_i$. For each $n\in\mathbb{N}$, we set $$
Q_n^{+} := \{ \alpha \in Q^+ \mid |\alpha| = n\}. $$

Let $A$ be a symmetrizable Cartan matrix. Let $\mathfrak{g}=\mathfrak{g}(A)$ be the corresponding Kac-Moody Lie algebra associated to $A$ with Cartan subalgebra $\mathfrak{h}:= \Q \otimes_{{\Z}}P^\vee$. Since $A$ is symmetrizable, there is a symmetric bilinear form $(,)$ on $\mathfrak{h}^*$ satisfying
\begin{align*}
  & (\alpha_i, \alpha_j) = d_i a_{ij} \quad (i,j \in I) \quad \text{and} \\
  & \< h_i, \Lam \> = \frac{2(\alpha_i,\Lam)}{(\alpha_i,\alpha_i)} \quad \text{for any }\Lam \in \mathfrak{h}^* \text{ and } i\in I.
\end{align*}


Let $K$ be a field of arbitrary characteristic. Let $u,v$ be two commuting indeterminates over $K$. We fix a matrix $(Q_{i,j})_{i,j\in I}$ in $K[u,v]$ such that
\begin{align*}
  & Q_{i,j}(u,v) = Q_{j,i}(v,u), \quad  Q_{i,i}(u,v) = 0, \\
  & Q_{i,j}(u,v)=\sum_{k,q\geqslant 0} c_{i,j,k,q}\,u^kv^q, \,\,\,\ \text{if}\ i\neq j.
\end{align*}
where $c_{i,j,-a_{ij},0}\in K^\times$, and $c_{i,j,k,q}\neq 0$ only if $2(\alpha_i,\alpha_j)=-(\alpha_i,\alpha_i)k - (\alpha_j,\alpha_j)q$.

Let $\Sym_n = \<s_1 , \ldots , s_{n-1}\>$ be the symmetric group on $\{1,2,\cdots,n\}$, where $s_i = (i,i+1)$.
Then $\Sym_n$ acts naturally on $I^n$ by places permutation. The orbits of this action is identified with element of $Q_n^{+}$. Then $I^\alpha:=\{\nu=(\nu_1,\cdots,\nu_n)\in I^n|\sum_{j=1}^{n}\alpha_{\nu_j}=\alpha\}$ is the orbit corresponding to $\alpha\in Q_n^+$.


\begin{dfn}\label{KLR}
 Let $\alpha\in Q_n^+$. The quiver Hecke algebra $\RR_\alpha$ associated with a Cartan datum $(A,P,\Pi,P^{\vee},\Pi^{\vee})$, $(Q_{i,j})_{i,j\in I}$ and $\alpha \in Q_n^+$ is the associative algebra over $K$ generated by $e(\nu)\, (\nu\in I^\alpha)$, $x_k\, (1\leqslant k \leqslant n)$, $\tau_l \, (1\leqslant l \leqslant n-1)$ satisfying the following defining relations:
  \begin{equation*}
    \begin{aligned}
      & e(\nu) e(\nu') = \delta_{\nu, \nu'} e(\nu), \ \
      \sum_{\nu \in I^{\alpha}}  e(\nu) = 1, \\
      & x_{k} x_{l} = x_{l} x_{k}, \ \ x_{k} e(\nu) = e(\nu) x_{k}, \\
      & \tau_{l} e(\nu) = e(s_{l}(\nu)) \tau_{l}, \ \ \tau_{k} \tau_{l} = \tau_{l} \tau_{k} \ \ \text{if} \ |k-l|>1, \\
      & \tau_{k}^2 e(\nu) = Q_{\nu_{k}, \nu_{k+1}} (x_{k}, x_{k+1})e(\nu), \\
      & (\tau_{k} x_{l} - x_{s_k(l)} \tau_{k}) e(\nu) = \begin{cases}
      -e(\nu) \ \ & \text{if} \ l=k, \nu_{k} = \nu_{k+1}, \\
      e(\nu) \ \ & \text{if} \ l=k+1, \nu_{k}=\nu_{k+1}, \\
      0 \ \ & \text{otherwise},
      \end{cases} \\[.5ex]
      & (\tau_{k+1} \tau_{k} \tau_{k+1}-\tau_{k} \tau_{k+1} \tau_{k}) e(\nu)\\
      &\hspace*{8ex} =\begin{cases} \dfrac{Q_{\nu_{k}, \nu_{k+1}}(x_{k},
      x_{k+1}) - Q_{\nu_{k+2}, \nu_{k+1}}(x_{k+2}, x_{k+1})}{x_{k} - x_{k+2}}e(\nu) \ \ & \text{if} \
      \nu_{k} = \nu_{k+2}, \\
      0 \ \ & \text{otherwise}.
      \end{cases}
    \end{aligned}
  \end{equation*}
\end{dfn}

In particular, $\RR_0 \cong K$.
The algebra $\RR_\alpha$ is ${\Z}$-graded, whose grading is uniquely determined by
\begin{equation*}
  \deg e(\nu) = 0, \qquad \deg x_k e(\nu) = (\alpha_{\nu_k},\alpha_{\nu_k}), \qquad \deg \tau_l e(\nu) = - (\alpha_{\nu_l},\alpha_{\nu_{l+1}}).
\end{equation*}

Let $\Lambda \in P^+$ be a dominant integral weight. We now recall the definition of the cyclotomic KLR algebra $\RR_\alpha^\Lambda$.
For $1\leqslant k \leqslant n$, we define
\begin{equation*}
  a^\Lambda_\alpha (x_k) = \sum_{\nu\in I^\alpha} x_k^{\<h_{\nu_k},\Lam\>}e(\nu) \in \RR_\alpha.
\end{equation*}

\begin{dfn}\label{cyclotomicKLR}
  Set $I_{\Lambda,\alpha} = \RR_\alpha a^\Lambda_\alpha (x_1) \RR_\alpha$. The {cyclotomic KLR algebra} $\RR^\Lambda_\alpha$ is defined to be the quotient algebra:   $$
  \RR_\alpha^\Lambda = \RR_\alpha/I_{\Lambda,\alpha}. $$
\end{dfn}
Sometimes we shall write $\R[\alpha](K)$ instead of $\R[\alpha]$ in order to emphasize the ground field $K$. In general, if $\O$ is a commutative ring and
$Q_{ij}(u,v)\in\O[u,v]$ for any $i,j\in I$, then we can define the cyclotomic KLR algebra $\R[\alpha](\O)$ over $\O$ in a similar way.
By construction, the cyclotomic KLR algebra $\R[\alpha]$ inherits a $\Z$-grading from $\RR_\alpha$. Let $\R[\alpha]\lgmod$ be the category of finite dimensional graded $\R[\alpha]$-modules.

\begin{dfn} We use $*$ to denote the unique $K$-linear anti-involution of $\R[\alpha]$ which is the identity map on all of its KLR generators: $$
e(\nu),\,\,\tau_r,\,\,x_k,\,\,\nu\in I^\alpha, 1\leq r<n, 1\leq k\leq n.
$$
\end{dfn}

Let $\circledast$ be the duality functor induced from $\tau$. That is, for any $M\in\R[\alpha]\lgmod$, $M^{\circledast}:=\Hom_K(M,K)$ as a $K$-linear space.
As a $\R[\alpha]$-module, $$
(a\cdot f)(x):=f(a^*x),\,\,\forall\,a\in\R[\alpha], f\in M^{\circledast}, x\in M .
$$
It is clear that $M^{\circledast}\in\R[\alpha]\lgmod$. We call ``$\circledast$'' the duality functor on $\R[\alpha]\lgmod$.

%



\begin{lem}[\text{\rm \cite[Proposition 3.10]{SVV}}]\label{tLam} The $K$-algebra $\R[\alpha]$ is equipped with a homogeneous symmetrizing form of degree $-d_{\Lam,\alpha}$. In particular, there is an $\R[\alpha]$-$\R[\alpha]$-bimodule isomorphism: $$
\bigl(\R[\alpha]\bigr)^{\circledast}\cong\R[\alpha]\<-d_{\Lam,\alpha}\>, $$
which induces a $K$-linear isomorphism $\Tr(\R[\alpha])\cong \bigl(Z(\R[\alpha])\bigr)^{\circledast}\<d_{\Lam,\alpha}\>$, where $d_{\Lam,\alpha}:=2(\alpha,\Lam)-(\alpha,\alpha)$.
\end{lem}

\begin{conv} Any product of the form $\prod_{j=a}^{b}A_j$ is understood as empty whenever $a>b$. Any increasing-index sequence expression of the form $\tau_j\tau_{j+1}\cdots \tau_k$ is understood as empty whenever $j>k$. Similarly, any decreasing-index sequence expression of the form $\tau_i\tau_{i-1}\cdots \tau_k$ is understood as empty whenever $i<k$.
\end{conv}
\bigskip

\section{Some nilHecke algebra calculations on cocenters}

In this section we shall do some subtle nilHecke algebra calculations on cocenters with a purpose to refine and generalize \cite[Lemma 3.17, Corollary 3.20]{HS3} to the ground field of arbitrary characteristic. The main new observation is that the images of all the primitive idempotents of the nilHecke algebra $\NH_n$ are the same in the cocenter, and thus one can regard each primitive idempotent as a common ``divided power'' of the identity element inside the cocenter of $\NH_n$. In general, we shall take this ``divided power'' in each ``nilHecke part'' into account when constructing a compact form of a $K$-linear generators of the cocenter of the cyclotomic KLR algebra $\R[\alpha]$.


\begin{dfn} Let $R$ be an integral domain. Let $n\in\Z^{\geq 1}$. The nilHecke algebra $\NH_n(R)$ of type $A$ is defined to be the unital $R$-algebra which can be presented by the generators $x_1,\cdots,x_n,\tau_1,\cdots,\tau_{n-1}$ and the following relations: $$\begin{aligned}
\tau_r^2=0,\,\,\forall\,1\leq r<n;\quad \tau_i\tau_j=\tau_j\tau_i,\,\,\forall\,1\leq i<j-1<n-1;\\
\tau_a\tau_{a+1}\tau_a=\tau_{a+1}\tau_a\tau_{a+1},\,\,\forall\,1\leq a<n-1;\\
x_jx_k=x_kx_j,\,\,1\leq j,k\leq n;\quad \tau_a x_b=x_b\tau_a,\,\,\forall\,b\notin\{a,a+1\};\\
\tau_j x_{j+1}=x_j\tau_j+1,\,\,x_{j+1}\tau_j=\tau_j x_j+1,\,\,\forall\,1\leq j<n-1 .
\end{aligned}
$$
\end{dfn}

Recall that $e_{[1,n]}:=\tau_{w_0}x_2x_3^{2}\cdots x_{n}^{n-1}\in \NH_{n}(\Z)$.

\begin{lem}\text{\rm (\cite[Lemma 2.19, Proposition 2.21]{Rou2}, \cite{KL1})}\label{idempn} The element $e_{[1,n]}$ is a degree $0$ homogeneous primitive idempotent in $\NH_{n}(\Z)$, and $1-e_{[1,n]}$ can be decomposed into a direct sum of $n!-1$ pairwise orthogonal primitive idempotents $e_{2,n},\cdots,e_{n!,n}$. Moreover, for any $2\leq j\leq n!$, $\NH_{n}(\Z)e_{j,n}\cong\NH_{n}(\Z) e_{[1,n]}$ is isomorphic to the unique indecomposable finitely generated projective $\NH_{n}(\Z)$-module $P(n)$.
\end{lem}

\begin{cor} We have $$
(n!)e_{[1,n]}\equiv 1\pmod{[\NH_n(\Z),\NH_n(\Z)]} .
$$
In particular, $e_{[1,n]}+[\NH_n(\Q),\NH_n(\Q)]$ is equal to the divided power $$\frac{1}{n!}+[\NH_n(\Q),\NH_n(\Q)]$$ in $\Tr(\NH_n(\Q))$.
\end{cor}

%

Let $1\leq u\leq v\leq n$, let $w[u,v]$ be the unique longest element in the symmetric group $\Sym_{\{u,u+1,\cdots,v\}}$. We use the reduced expression $s_u s_{u+1}\cdots s_{v-1} s_{u} s_{u+1}\cdots s_{v-2}\cdots s_{u}$ of $w[u,v]$ to define $\tau_{w[u,v]}:=\tau_u\tau_{u+1}\cdots\tau_{v-1}\tau_{u}\tau_{u+1}\cdots\tau_{v-2}\cdots\tau_{u}$, and set $$
e_{[u,v]}:=\begin{cases}\tau_{w[u,v]}x_{u+1}x_{u+2}^2\cdots x_{v}^{v-u} &\text{if $u<v$;}\\
1 &\text{if $u=v$.}
\end{cases}$$
We regard $e_{[u,v]}$ as an element in $\NH_n(\Z)$. Let $\alpha\in Q_n^+$ and $\nu\in I^\alpha$ such that $\nu_u=\nu_{u+1}=\cdots=\nu_v$. Then we regard $e_{[u,v]}e(\nu)$ as an element in the KLR algebra $\RR_\alpha$ (or the cyclotomic KLR algebra $\R[\alpha]$).

\begin{dfn} Let $k\geq 1$, $1\leq t\leq n-2$ and $t+2\leq l\leq n$. We define $$\begin{aligned}
X_{k,t,l}&:=\bigl(\prod_{j=2}^{t+1}x_j^{j-1}\bigr)\bigl(\prod_{j=t+2}^{l-1}x_j^{t+1}\bigr)\bigl(\prod_{j=l}^{n}x_j^{t}\bigr)x_n^{k-1}\times(\tau_1\cdots \tau_{n-1})\times(\tau_1\cdots \tau_{n-2})\\
&\qquad \times\cdots\times (\tau_1\cdots \tau_{n-t})\times (\tau_1\tau_2\cdots \tau_{l-(t+2)}),
\end{aligned}
$$
\end{dfn}

\begin{lem}\label{xlb} Let $n>2$, $k\geq 1$, $1\leq t\leq n-2$, $a\geq 2$ and $a(t+1)<l\leq n$.  Then for any $1\leq b\leq a$, we have $$
x_{l-b(t+1)}X_{k,t,l}\tau_{l-(t+1)}\equiv x_{l-(t+1)}X_{k,t,l}\tau_{l-(t+1)}\pmod{[\NH_n(\Z),\NH_n(\Z)]}.
$$
\end{lem}

\begin{proof}We set $A_2:=x_{l-(t+1)}X_{k,t,l}\tau_{l-(t+1)}$. Assume that $A_2\equiv x_{l-b(t+1)}X_{k,t,l}\tau_{l-(t+1)}\pmod{[\NH_n(\Z),\NH_n(\Z)]}$ for some $1\leq b<a$. Then $$\begin{aligned}
A_2&\equiv\bigl(\prod_{j=2}^{t+1}x_j^{j-1}\bigr)\bigl(\prod_{j=t+2}^{l-1}x_j^{t+1}\bigr)\bigl(\prod_{j=l}^{n}x_j^{t}\bigr)x_n^{k-1}(\tau_1\cdots\tau_{l-b(t+1)-2}x_{l-b(t+1)} \tau_{l-b(t+1)-1}\cdots \tau_{n-1})\\
&\qquad(\tau_1\cdots \tau_{n-2})\cdots(\tau_1\cdots\tau_{n-t})(\tau_1\cdots\tau_{l-(t+1)})\\
&=\bigl(\prod_{j=2}^{t+1}x_j^{j-1}\bigr)\bigl(\prod_{j=t+2}^{l-1}x_j^{t+1}\bigr)\bigl(\prod_{j=l}^{n}x_j^{t}\bigr)x_n^{k-1}(\tau_1\cdots\tau_{l-b(t+1)-2}(x_{l-b(t+1)}\tau_{l-b(t+1)-1})\tau_{l-b(t+1)}\cdots \tau_{n-1})\\
&\qquad(\tau_1\cdots \tau_{n-2})\cdots(\tau_1\cdots\tau_{n-t})(\tau_1\cdots\tau_{l-(t+1)})\\
&=\bigl(\prod_{j=2}^{t+1}x_j^{j-1}\bigr)\bigl(\prod_{j=t+2}^{l-1}x_j^{t+1}\bigr)\bigl(\prod_{j=l}^{n}x_j^{t}\bigr)x_n^{k-1}(\tau_1\cdots \tau_{l-b(t+1)-2}(\tau_{l-b(t+1)-1}x_{l-b(t+1)-1}+1)\tau_{l-b(t+1)}\cdots \tau_{n-1})\\
&\qquad(\tau_1\cdots \tau_{n-2})\cdots(\tau_1\cdots\tau_{n-t})(\tau_1\cdots\tau_{l-(t+1)})\\
&=\bigl(\prod_{j=2}^{t+1}x_j^{j-1}\bigr)\bigl(\prod_{j=t+2}^{l-1}x_j^{t+1}\bigr)\bigl(\prod_{j=l}^{n}x_j^{t}\bigr)x_n^{k-1}\biggl\{\tau_1\cdots \tau_{l-b(t+1)-2}\tau_{l-b(t+1)-1}x_{l-b(t+1)-1}\tau_{l-b(t+1)}\cdots \tau_{n-1}\\
&\qquad +(\tau_1\cdots \tau_{l-b(t+1)-2})(\tau_{l-b(t+1)}\cdots \tau_{n-1})\biggr\}(\tau_1\cdots \tau_{n-2})\cdots(\tau_1\cdots\tau_{n-t})(\tau_1\cdots\tau_{l-(t+1)})\\
&=\bigl(\prod_{j=2}^{t+1}x_j^{j-1}\bigr)\bigl(\prod_{j=t+2}^{l-1}x_j^{t+1}\bigr)\bigl(\prod_{j=l}^{n}x_j^{t}\bigr)x_n^{k-1}\biggl\{\tau_1\cdots \tau_{l-b(t+1)-2}\tau_{l-b(t+1)-1}x_{l-b(t+1)-1}\tau_{l-b(t+1)}\cdots \tau_{n-1}\\
&\qquad +(\tau_{l-b(t+1)}\cdots \tau_{n-1})(\tau_1\cdots \tau_{l-b(t+1)-2})\biggr\}(\tau_{1}\cdots\tau_{n-2})\cdots(\tau_1\cdots\tau_{n-t})(\tau_1\cdots\tau_{l-(t+1)})\\
&=\bigl(\prod_{j=2}^{t+1}x_j^{j-1}\bigr)\bigl(\prod_{j=t+2}^{l-1}x_j^{t+1}\bigr)\bigl(\prod_{j=l}^{n}x_j^{t}\bigr)x_n^{k-1}(\tau_1\cdots \tau_{n-1})\bigl(\tau_1\cdots\tau_{l-b(t+1)-3}x_{l-b(t+1)-1}\tau_{l-b(t+1)-2}
\cdots\tau_{n-2}\bigr)\\
&\qquad (\tau_1\cdots \tau_{n-3})\cdots(\tau_1\cdots\tau_{n-t})(\tau_1\cdots\tau_{l-(t+1)})\\
&\quad\vdots\\
&=\bigl(\prod_{j=2}^{t+1}x_j^{j-1}\bigr)\bigl(\prod_{j=t+2}^{l-1}x_j^{t+1}\bigr)\bigl(\prod_{j=l}^{n}x_j^{t}\bigr)x_n^{k-1}(\tau_1\cdots \tau_{n-1})\cdots(\tau_1\cdots \tau_{n-t})x_{l-(b+1)t-b}(\tau_1\cdots\tau_{l-(t+1)})\\
&=\bigl(\prod_{j=2}^{t+1}x_j^{j-1}\bigr)\bigl(\prod_{j=t+2}^{l-1}x_j^{t+1}\bigr)\bigl(\prod_{j=l}^{n}x_j^{t}\bigr)x_n^{k-1}(\tau_1\cdots \tau_{n-1})(\tau_1\cdots \tau_{n-2})\cdots(\tau_1\cdots \tau_{n-t})\\
&\qquad (\tau_1\cdots \tau_{l-(b+1)t-b-2}x_{l-(b+1)t-b}\tau_{l-(b+1)t-b-1}\cdots \tau_{l-(t+1)})\\
&=\bigl(\prod_{j=2}^{t+1}x_j^{j-1}\bigr)\bigl(\prod_{j=t+2}^{l-1}x_j^{t+1}\bigr)\bigl(\prod_{j=l}^{n}x_j^{t}\bigr)x_n^{k-1}(\tau_1\cdots \tau_{n-1})(\tau_1\cdots \tau_{n-2})\cdots(\tau_1\cdots \tau_{n-t})\\
&\qquad \bigl(\tau_1\cdots \tau_{l-(b+1)t-b-2}(\tau_{l-(b+1)t-b-1}x_{l-(b+1)t-b-1}+1)\tau_{l-(b+1)t-b}\cdots \tau_{l-(t+1)}\bigr)\\
&=\bigl(\prod_{j=2}^{t+1}x_j^{j-1}\bigr)\bigl(\prod_{j=t+2}^{l-1}x_j^{t+1}\bigr)\bigl(\prod_{j=l}^{n}x_j^{t}\bigr)x_n^{k-1}(\tau_1\cdots \tau_{n-1})(\tau_1\cdots \tau_{n-2})\cdots(\tau_1\cdots \tau_{n-t})\\
&\qquad \biggl\{(\tau_1\cdots \tau_{l-(t+1)})x_{l-(b+1)(t+1)}+(\tau_1\cdots \tau_{l-(b+1)t-b-2}\tau_{l-(b+1)t-b}\cdots \tau_{l-(t+1)})\biggr\}\\
&\equiv  x_{l-(b+1)(t+1)}X_{k,t,l}\tau_{l-(t+1)}+\tau_{l-t-1}\bigl(\prod_{j=2}^{t+1}x_j^{j-1}\bigr)\bigl(\prod_{j=t+2}^{l-1}x_j^{t+1}\bigr)\bigl(\prod_{j=l}^{n}x_j^{t}\bigr)x_n^{k-1}\\
&\qquad (\tau_1\cdots \tau_{n-1})(\tau_1\cdots \tau_{n-2})\cdots(\tau_1\cdots \tau_{n-t})(\tau_1\cdots \tau_{l-(b+1)t-b-2}\tau_{l-(b+1)t-b}\cdots \tau_{l-t-2})\\
&\quad\pmod{[\NH_n(\Z),\NH_n(\Z)]}.
\end{aligned} $$

Let $B_2$ be the second term in the above right-hand side sum. Since $t+2\leq l-t-1\leq l-2$, we have \begin{align*}
B_2&:=\tau_{l-t-1}\bigl(\prod_{j=2}^{t+1}x_j^{j-1}\bigr)\bigl(\prod_{j=t+2}^{l-1}x_j^{t+1}\bigr)\bigl(\prod_{j=l}^{n}x_j^{t}\bigr)x_n^{k-1}(\tau_1\cdots \tau_{n-1})(\tau_1\cdots \tau_{n-2})\cdots(\tau_1\cdots \tau_{n-t})\\
&\qquad (\tau_1\cdots \tau_{l-(b+1)t-b-2}\tau_{l-(b+1)t-b}\cdots \tau_{l-t-2})\\
&=\bigl(\prod_{j=2}^{t+1}x_j^{j-1}\bigr)\bigl(\prod_{j=t+2}^{l-1}x_j^{t+1}\bigr)\bigl(\prod_{j=l}^{n}x_j^{t}\bigr)x_n^{k-1}\tau_{l-t-1}(\tau_1\cdots \tau_{n-1})(\tau_1\cdots \tau_{n-2})\cdots(\tau_1\cdots \tau_{n-t})\\
&\qquad (\tau_1\cdots \tau_{l-(b+1)t-b-2}\tau_{l-(b+1)t-b}\cdots \tau_{l-t-2})\\
&=\bigl(\prod_{j=2}^{t+1}x_j^{j-1}\bigr)\bigl(\prod_{j=t+2}^{l-1}x_j^{t+1}\bigr)\bigl(\prod_{j=l}^{n}x_j^{t}\bigr)x_n^{k-1}(\tau_1\cdots \tau_{n-1})\tau_{l-t-2}(\tau_1\cdots \tau_{n-2})\cdots(\tau_1\cdots \tau_{n-t})\\
&\qquad (\tau_1\cdots \tau_{l-(b+1)t-b-2}\tau_{l-(b+1)t-b}\cdots \tau_{l-t-2})\\
&\quad\vdots\\
&=\bigl(\prod_{j=2}^{t+1}x_j^{j-1}\bigr)\bigl(\prod_{j=t+2}^{l-1}x_j^{t+1}\bigr)\bigl(\prod_{j=l}^{n}x_j^{t}\bigr)x_n^{k-1}(\tau_1\cdots \tau_{n-1})(\tau_1\cdots \tau_{n-2})\cdots(\tau_1\cdots \tau_{n-t})\\
&\qquad \tau_{l-2t-1}(\tau_1\cdots \tau_{l-(b+1)t-b-2}\tau_{l-(b+1)t-b}\cdots \tau_{l-t-2})\\
&=\bigl(\prod_{j=2}^{t+1}x_j^{j-1}\bigr)\bigl(\prod_{j=t+2}^{l-1}x_j^{t+1}\bigr)\bigl(\prod_{j=l}^{n}x_j^{t}\bigr)x_n^{k-1}(\tau_1\cdots \tau_{n-1})(\tau_1\cdots \tau_{n-2})\cdots(\tau_1\cdots \tau_{n-t})\\
&\qquad (\tau_1\cdots \tau_{l-(b+1)t-b-2}\tau_{l-(b+1)t-b}\cdots \tau_{l-t-2})\tau_{l-2t-2}\\
&\equiv\tau_{l-2t-2}\bigl(\prod_{j=2}^{t+1}x_j^{j-1}\bigr)\bigl(\prod_{j=t+2}^{l-1}x_j^{t+1}\bigr)\bigl(\prod_{j=l}^{n}x_j^{t}\bigr)x_n^{k-1}(\tau_1\cdots \tau_{n-1})(\tau_1\cdots \tau_{n-2})\cdots(\tau_1\cdots \tau_{n-t})\\
&\qquad (\tau_1\cdots \tau_{l-(b+1)t-b-2}\tau_{l-(b+1)t-b}\cdots \tau_{l-t-2})\\
&\quad\vdots\\
&=\bigl(\prod_{j=2}^{t+1}x_j^{j-1}\bigr)\bigl(\prod_{j=t+2}^{l-1}x_j^{t+1}\bigr)\bigl(\prod_{j=l}^{n}x_j^{t}\bigr)x_n^{k-1}(\tau_1\cdots \tau_{n-1})(\tau_1\cdots \tau_{n-2})\cdots(\tau_1\cdots \tau_{n-t})\\
&\qquad \tau_{l-(b+1)t-b}(\tau_1\cdots \tau_{l-(b+1)t-b-2}\tau_{l-(b+1)t-b}\cdots \tau_{l-t-2})\equiv 0\pmod{[\NH_n(\Z),\NH_n(\Z)]},
\end{align*}
where we have used the fact that $\tau_{l-t-1}$ commutes with $(x_{l-t-1}x_{l-t})^{t+1}$ in the second equality above, and used the fact that $l-bt-b\geq t+2$ in the last equality. As a result, $$\begin{aligned}
A_2\equiv x_{l-(b+1)(t+1)}X_{k,t,l}\tau_{l-(t+1)}\pmod{[\NH_n(\Z),\NH_n(\Z)]}.
\end{aligned} $$
This proves the lemma.
\end{proof}

For any two pairs of integers $(a,b),(a',b')$, we write $(a,b)\leq (a',b')$ whenever either $a<a'$ or $a=a'$ and $b\leq b'$ (i.e. we use the lexicographic order). If $(a,b)\leq (a',b')$ and $(a,b)\neq (a',b')$, then we write $(a,b)<(a',b')$.

\begin{prop}\label{pre nil-Hecke relation}
Suppose $n\geq 2$ and $k\geq 1$. Let $$
X_{k,n}:=x_2x_3\cdots x_{n-1}x_n^k \tau_1\tau_2\cdots \tau_{n-1}\in\NH_n(\Z). $$
Then we have $X_{k,n}-e_{[1,n]}x_n^{k-1} \in [\NH_n(\Z),\NH_n(\Z)]$.
\end{prop}

\begin{proof} We claim that for any $1\leq t\leq n-2$ and $t+2\leq l\leq n$, \begin{equation}\label{induction1}
X_{k,n}\equiv X_{k,t,l} \pmod{[\NH_n(\Z),\NH_n(\Z)]}.
\end{equation}
We use induction on $(t,l)$ to prove (\ref{induction1}). If $t=1$ and $l=t+2=3$, it is trivial to see that $X_{k,n}=X_{k,1,3}$. Suppose (\ref{induction1}) holds for pair $(t,l)$ with $t<n-2$. We divide the proof into two cases.

\medskip
{\it Case 1.} Assume $t+2=l$ and $t<n-2$. Then we have $$\begin{aligned}
&\quad\, X_{k,n}\equiv X_{k,t,t+2}\\
&\equiv \bigl(\prod_{j=2}^{t+1}x_j^{j-1}\bigr)\bigl(\prod_{j=t+2}^{n}x_j^{t}\bigr)x_n^{k-1}(\tau_1\cdots \tau_{n-1})(\tau_1\cdots \tau_{n-2})\cdots (\tau_1\cdots \tau_{n-t})\\
&\qquad\qquad \pmod{[\NH_n(\Z),\NH_n(\Z)]}.
\end{aligned}$$

We can compute $$\begin{aligned}
&\quad\,\bigl(\prod_{j=2}^{t+1}x_j^{j-1}\bigr)\bigl(\prod_{j=t+2}^{n}x_j^{t}\bigr)x_n^{k-1}(\tau_1\cdots \tau_{n-1})(\tau_1\cdots \tau_{n-2})\cdots (\tau_1\cdots \tau_{n-t})\\
&=\bigl(\prod_{j=2}^{t+1}x_j^{j-1}\bigr)\bigl(\prod_{j=t+2}^{n}x_j^{t}\bigr)x_n^{k-1}(\tau_1\cdots \tau_{n-1})(\tau_1\cdots \tau_{n-2})\cdots (\tau_1\cdots \tau_{n-t})(x_2\tau_1-\tau_1x_1)\\
&\equiv \bigl(\prod_{j=2}^{t+1}x_j^{j-1}\bigr)\bigl(\prod_{j=t+2}^{n}x_j^{t}\bigr)x_n^{k-1}(\tau_1\cdots \tau_{n-1})(\tau_1\cdots \tau_{n-2})\cdots (\tau_1\cdots \tau_{n-t})x_2\tau_1\\
&\quad\,-\tau_1x_1\bigl(\prod_{j=2}^{t+1}x_j^{j-1}\bigr)\bigl(\prod_{j=t+2}^{n}x_j^{t}\bigr)x_n^{k-1}(\tau_1\cdots \tau_{n-1})(\tau_1\cdots \tau_{n-2})\cdots (\tau_1\cdots \tau_{n-t})\\
&\equiv \bigl(\prod_{j=2}^{t+1}x_j^{j-1}\bigr)\bigl(\prod_{j=t+2}^{n}x_j^{t}\bigr)x_n^{k-1}(\tau_1\cdots \tau_{n-1})(\tau_1\cdots \tau_{n-2})\cdots (\tau_1\cdots \tau_{n-t})x_2\tau_1\\
&\quad\,-x_1\bigl(\prod_{j=2}^{t+1}x_j^{j-1}\bigr)\bigl(\prod_{j=t+2}^{n}x_j^{t}\bigr)x_n^{k-1}(\tau_1\tau_1\cdots \tau_{n-1})(\tau_1\cdots \tau_{n-2})\cdots (\tau_1\cdots \tau_{n-t})\\
&\equiv \bigl(\prod_{j=2}^{t+1}x_j^{j-1}\bigr)\bigl(\prod_{j=t+2}^{n}x_j^{t}\bigr)x_n^{k-1}(\tau_1\cdots \tau_{n-1})(\tau_1\cdots \tau_{n-2})\cdots (\tau_1\cdots \tau_{n-t})x_2\tau_1\\
&\qquad\pmod{[\NH_n(\Z),\NH_n(\Z)]},
\end{aligned}$$
where we have used the fact that $\tau_1$ commutes with $x_1x_2$ in the second last equality.

In a similar way, we can get that
$$\begin{aligned}
&\quad\, \bigl(\prod_{j=2}^{t+1}x_j^{j-1}\bigr)\bigl(\prod_{j=t+2}^{n}x_j^{t}\bigr)x_n^{k-1}(\tau_1\cdots \tau_{n-1})(\tau_1\cdots \tau_{n-2})\cdots (\tau_1\cdots \tau_{n-t})x_2\tau_1\\
&=\bigl(\prod_{j=2}^{t+1}x_j^{j-1}\bigr)\bigl(\prod_{j=t+2}^{n}x_j^{t}\bigr)x_n^{k-1}(\tau_1\cdots \tau_{n-1})(\tau_1\cdots \tau_{n-2})\cdots (\tau_1\tau_2x_2\tau_3\cdots \tau_{n-t}\tau_1)\\
&=\bigl(\prod_{j=2}^{t+1}x_j^{j-1}\bigr)\bigl(\prod_{j=t+2}^{n}x_j^{t}\bigr)x_n^{k-1}(\tau_1\cdots \tau_{n-1})(\tau_1\cdots \tau_{n-2})\cdots \bigl(\tau_1(x_3\tau_2-1)\tau_3\cdots \tau_{n-t}\tau_1\bigr)\\
&=\bigl(\prod_{j=2}^{t+1}x_j^{j-1}\bigr)\bigl(\prod_{j=t+2}^{n}x_j^{t}\bigr)x_n^{k-1}(\tau_1\cdots \tau_{n-1})(\tau_1\cdots \tau_{n-2})\cdots \bigl(\tau_1x_3\tau_2\tau_3\cdots \tau_{n-t}\tau_1\bigr)\\
&=\bigl(\prod_{j=2}^{t+1}x_j^{j-1}\bigr)\bigl(\prod_{j=t+2}^{n}x_j^{t}\bigr)x_n^{k-1}(\tau_1\cdots \tau_{n-1})(\tau_1\cdots \tau_{n-2})\cdots
\bigl(\tau_1\tau_2(\tau_3x_3)\tau_4\cdots \tau_{n-t+1}\bigr)\bigl(\tau_2\tau_1\tau_2\tau_3\cdots \tau_{n-t}\bigr)\\
&=\bigl(\prod_{j=2}^{t+1}x_j^{j-1}\bigr)\bigl(\prod_{j=t+2}^{n}x_j^{t}\bigr)x_n^{k-1}(\tau_1\cdots \tau_{n-1})(\tau_1\cdots \tau_{n-2})\cdots
\bigl(\tau_1\tau_2(x_4\tau_3-1)\tau_4\cdots \tau_{n-t+1}\bigr)\bigl(\tau_2\tau_1\tau_2\tau_3\cdots \tau_{n-t}\bigr)\\
&\quad\vdots\\
&=\bigl(\prod_{j=2}^{t+1}x_j^{j-1}\bigr)\bigl(\prod_{j=t+2}^{n}x_j^{t}\bigr)x_n^{k-1}x_{t+2}\tau_{t+1}(\tau_1\cdots \tau_{n-1})(\tau_1\cdots \tau_{n-2})\cdots
\bigl(\tau_1\tau_2\tau_3\tau_4\cdots \tau_{n-t+1}\bigr)\bigl(\tau_1\tau_2\tau_3\cdots \tau_{n-t}\bigr)\\
&=\bigl(\prod_{j=2}^{t+1}x_j^{j-1}\bigr)x_{t+2}^{t+1}\bigl(\prod_{j=t+3}^{n}x_j^{t}\bigr)x_n^{k-1}(\tau_1\cdots \tau_{n-1})(\tau_1\cdots \tau_{n-2})\cdots
\bigl(\tau_1\tau_2\tau_3\tau_4\cdots \tau_{n-t+1}\bigr)\bigl(\tau_1\tau_2\tau_3\cdots \tau_{n-t}\bigr)\tau_1\\
&\equiv X_{k,t,t+3}\pmod{[\NH_n(\Z),\NH_n(\Z)].}\\
\end{aligned}$$

\medskip
{\it Case 2.} Assume $t+2<l\leq n$ and $t<n-2$. In this case, we have $$\begin{aligned}
&\quad\,X_{k,n}\equiv X_{k,t,l}\\
&=\bigl(\prod_{j=2}^{t+1}x_j^{j-1}\bigr)\bigl(\prod_{j=t+2}^{l-1}x_j^{t+1}\bigr)\bigl(\prod_{j=l}^{n}x_j^{t}\bigr)x_n^{k-1}
(\tau_1\cdots \tau_{n-1})(\tau_1\cdots \tau_{n-2})\cdots(\tau_1\cdots \tau_{n-t})(\tau_1\cdots \tau_{l-(t+2)})\\
&=\bigl(\prod_{j=2}^{t+1}x_j^{j-1}\bigr)\bigl(\prod_{j=t+2}^{l-1}x_j^{t+1}\bigr)\bigl(\prod_{j=l}^{n}x_j^{t}\bigr)x_n^{k-1}
(\tau_1\cdots \tau_{n-1})(\tau_1\cdots \tau_{n-2})\cdots(\tau_1\cdots \tau_{n-t})(\tau_1\cdots \tau_{l-(t+2)})\\
&\qquad (x_{l-t}\tau_{l-(t+1)}-\tau_{l-(t+1)}x_{l-(t+1)})\\
&\equiv \bigl(\prod_{j=2}^{t+1}x_j^{j-1}\bigr)\bigl(\prod_{j=t+2}^{l-1}x_j^{t+1}\bigr)\bigl(\prod_{j=l}^{n}x_j^{t}\bigr)x_n^{k-1}
(\tau_1\cdots \tau_{n-1})(\tau_1\cdots \tau_{n-2})\cdots(\tau_1\cdots\tau_{l-t}x_{l-t}\tau_{l-t+1}\cdots\tau_{n-t})\\
&\qquad (\tau_1\cdots \tau_{l-(t+1)})-x_{l-(t+1)}\bigl(\prod_{j=2}^{t+1}x_j^{j-1}\bigr)\bigl(\prod_{j=t+2}^{l-1}x_j^{t+1}\bigr)\bigl(\prod_{j=l}^{n}x_j^{t}\bigr)x_n^{k-1}(\tau_1\cdots \tau_{n-1})(\tau_1\cdots \tau_{n-2})\\
&\qquad\qquad\cdots(\tau_1\cdots\tau_{n-t})(\tau_1\cdots\tau_{l-(t+1)})\pmod{[\NH_n(\Z),\NH_n(\Z)]}.
\end{aligned}$$

We denote by $A_1, A_2$ the first term and the second term in the above last equality respectively. That is, $$
X_{k,n}\equiv X_{k,t,l}\equiv A_1-A_2\pmod{[\NH_n(\Z),\NH_n(\Z)]} .
$$
Now $$\begin{aligned}
A_1 &\equiv \bigl(\prod_{j=2}^{t+1}x_j^{j-1}\bigr)\bigl(\prod_{j=t+2}^{l-1}x_j^{t+1}\bigr)\bigl(\prod_{j=l}^{n}x_j^{t}\bigr)x_n^{k-1}(\tau_1\cdots \tau_{n-1})(\tau_1\cdots \tau_{n-2})\cdots \\
&\qquad (\tau_1\cdots \tau_{l-t}x_{l-t}\tau_{l-t+1}\cdots \tau_{n-t})(\tau_1\cdots \tau_{l-(t+1)})\\
&=\bigl(\prod_{j=2}^{t+1}x_j^{j-1}\bigr)\bigl(\prod_{j=t+2}^{l-1}x_j^{t+1}\bigr)\bigl(\prod_{j=l}^{n}x_j^{t}\bigr)x_n^{k-1}(\tau_1\cdots \tau_{n-1})(\tau_1\cdots \tau_{n-2})\cdots\\
 &\qquad \bigl(\tau_1\cdots\tau_{l-t-1}(x_{l-t+1}\tau_{l-t}-1)\tau_{l-t+1}\cdots \tau_{n-t}\bigr)(\tau_1\cdots \tau_{l-(t+1)})\\
&=\bigl(\prod_{j=2}^{t+1}x_j^{j-1}\bigr)\bigl(\prod_{j=t+2}^{l-1}x_j^{t+1}\bigr)\bigl(\prod_{j=l}^{n}x_j^{t}\bigr)x_n^{k-1}(\tau_1\cdots \tau_{n-1})(\tau_1\cdots \tau_{n-2})\cdots \\
&\qquad (\tau_1\cdots \tau_{l-t}x_{l-t}\tau_{l-t+1}\cdots \tau_{n-t})(\tau_1\cdots \tau_{l-(t+1)})\\
&\,\,\,-\bigl(\prod_{j=2}^{t+1}x_j^{j-1}\bigr)\bigl(\prod_{j=t+2}^{l-1}x_j^{t+1}\bigr)\bigl(\prod_{j=l}^{n}x_j^{t}\bigr)x_n^{k-1}(\tau_1\cdots \tau_{n-1})(\tau_1\cdots \tau_{n-2})\cdots \\
&\qquad(\tau_1\cdots \tau_{l-t-1}\tau_{l-t+1}\cdots\tau_{n-t})(\tau_1\cdots \tau_{l-(t+1)})\\
&=\bigl(\prod_{j=2}^{t+1}x_j^{j-1}\bigr)\bigl(\prod_{j=t+2}^{l-1}x_j^{t+1}\bigr)\bigl(\prod_{j=l}^{n}x_j^{t}\bigr)x_n^{k-1}(\tau_1\cdots \tau_{n-1})(\tau_1\cdots \tau_{n-2})\cdots \\
&\qquad (\tau_1\cdots \tau_{l-t}x_{l-t}\tau_{l-t+1}\cdots \tau_{n-t})(\tau_1\cdots \tau_{l-(t+1)})\\
&\,\,\,-\bigl(\prod_{j=2}^{t+1}x_j^{j-1}\bigr)\bigl(\prod_{j=t+2}^{l-1}x_j^{t+1}\bigr)\bigl(\prod_{j=l}^{n}x_j^{t}\bigr)x_n^{k-1}(\tau_1\cdots \tau_{n-1})(\tau_1\cdots \tau_{n-2})\cdots \\
&\qquad(\tau_{l-t+1}\cdots\tau_{n-t})(\tau_1\cdots \tau_{l-t-1})(\tau_1\cdots \tau_{l-(t+1)})\\
&=\bigl(\prod_{j=2}^{t+1}x_j^{j-1}\bigr)\bigl(\prod_{j=t+2}^{l-1}x_j^{t+1}\bigr)\bigl(\prod_{j=l}^{n}x_j^{t}\bigr)x_n^{k-1}(\tau_1\cdots \tau_{n-1})(\tau_1\cdots \tau_{n-2})\cdots \\
&\qquad (\tau_1\cdots \tau_{l-t}x_{l-t}\tau_{l-t+1}\cdots \tau_{n-t})(\tau_1\cdots \tau_{l-(t+1)}), \end{aligned} $$
where in the last equality we used the fact that $$
(\tau_1\cdots \tau_{l-t-1})(\tau_1\cdots \tau_{l-(t+1)})=(\tau_2\cdots\tau_{l-t-1})(\tau_1\cdots \tau_{l-t-1})\tau_{l-(t+1)}=0 .
$$
Thus, in a similar way we can get that $$\begin{aligned}
A_1&=\bigl(\prod_{j=2}^{t+1}x_j^{j-1}\bigr)\bigl(\prod_{j=t+2}^{l-1}x_j^{t+1}\bigr)\bigl(\prod_{j=l}^{n}x_j^{t}\bigr)x_n^{k-1}(\tau_1\cdots \tau_{n-1})(\tau_1\cdots \tau_{n-2})\cdots \\
&\qquad (\tau_1\cdots \tau_{l-t}x_{l-t}\tau_{l-t+1}\cdots \tau_{n-t})(\tau_1\cdots \tau_{l-(t+1)})\\
&=\bigl(\prod_{j=2}^{t+1}x_j^{j-1}\bigr)\bigl(\prod_{j=t+2}^{l-1}x_j^{t+1}\bigr)\bigl(\prod_{j=l}^{n}x_j^{t}\bigr)x_n^{k-1}(\tau_1\cdots \tau_{n-1})(\tau_1\cdots \tau_{n-2})\cdots \\
&\qquad (\tau_1\cdots \tau_{l-t}(x_{l-t}\tau_{l-t+1})\cdots \tau_{n-t})(\tau_1\cdots \tau_{l-(t+1)})\\
&=\bigl(\prod_{j=2}^{t+1}x_j^{j-1}\bigr)\bigl(\prod_{j=t+2}^{l-1}x_j^{t+1}\bigr)\bigl(\prod_{j=l}^{n}x_j^{t}\bigr)x_n^{k-1}(\tau_1\cdots \tau_{n-1})(\tau_1\cdots \tau_{n-2})\cdots \\
&\qquad \bigl(\tau_1\cdots \tau_{l-t-1}(x_{l-t+1}\tau_{l-t}-1)\tau_{l-t+1}\cdots \tau_{n-t}\bigr)(\tau_1\cdots \tau_{l-(t+1)})\end{aligned} $$

$$\begin{aligned}
&=\bigl(\prod_{j=2}^{t+1}x_j^{j-1}\bigr)\bigl(\prod_{j=t+2}^{l-1}x_j^{t+1}\bigr)\bigl(\prod_{j=l}^{n}x_j^{t}\bigr)x_n^{k-1}(\tau_1\cdots \tau_{n-1})(\tau_1\cdots \tau_{n-2})\cdots \\
&\qquad \bigl(\tau_1\cdots \tau_{l-t-1}x_{l-t+1}\tau_{l-t}\tau_{l-t+1}\cdots \tau_{n-t}\bigr)(\tau_1\cdots \tau_{l-(t+1)})\\
&=\bigl(\prod_{j=2}^{t+1}x_j^{j-1}\bigr)\bigl(\prod_{j=t+2}^{l-1}x_j^{t+1}\bigr)\bigl(\prod_{j=l}^{n}x_j^{t}\bigr)x_n^{k-1}(\tau_1\cdots \tau_{n-1})(\tau_1\cdots \tau_{n-2})\cdots \\
&\qquad \bigl(\tau_1\cdots \tau_{l-t+1}x_{l-t+1}\tau_{l-t+2}\cdots \tau_{n-t+1}\bigr)(\tau_1\tau_2\cdots \tau_{n-t})(\tau_1\cdots \tau_{l-(t+1)})\\
&\quad\vdots\\
&=\bigl(\prod_{j=2}^{t+1}x_j^{j-1}\bigr)\bigl(\prod_{j=t+2}^{l-1}x_j^{t+1}\bigr)\bigl(\prod_{j=l}^{n}x_j^{t}\bigr)x_n^{k-1}x_l(\tau_1\cdots \tau_{n-1})(\tau_1\cdots \tau_{n-2})\cdots \\
&\qquad \bigl(\tau_1\cdots\cdots \tau_{n-t}\bigr)(\tau_1\cdots \tau_{l-(t+1)})\\
&=\bigl(\prod_{j=2}^{t+1}x_j^{j-1}\bigr)\bigl(\prod_{j=t+2}^{l}x_j^{t+1}\bigr)\bigl(\prod_{j=l+1}^{n}x_j^{t}\bigr)x_n^{k-1}(\tau_1\cdots \tau_{n-1})(\tau_1\cdots \tau_{n-2})\cdots \\
&\qquad \bigl(\tau_1\cdots\cdots \tau_{n-t}\bigr)(\tau_1\cdots \tau_{l-(t+1)}).
\end{aligned}$$
By definition, if $l<n$, then the above equality means $A_1\equiv X_{k,t,l+1}\pmod{[\NH_n(\Z),\NH_n(\Z)]}$; while if $l=n$, the above equality means $A_1\equiv X_{k,t+1,t+3}\pmod{[\NH_n(\Z),\NH_n(\Z)]}$.
It remains to show that \begin{equation}\label{a2}
\begin{matrix} A_2:=x_{l-(t+1)}X_{k,t,l}\tau_{l-(t+1)}\in [\NH_n(\Z),\NH_n(\Z)].\end{matrix}
\end{equation}

To this end, we first assume $l-(t+1)\leq (t+1)$. In this case, we have
\begin{align*}
A_2&:=x_{l-(t+1)}\bigl(\prod_{j=2}^{t+1}x_j^{j-1}\bigr)\bigl(\prod_{j=t+2}^{l-1}x_j^{t+1}\bigr)\bigl(\prod_{j=l}^{n}x_j^{t}\bigr)x_n^{k-1}(\tau_1\cdots \tau_{n-1})(\tau_1\cdots \tau_{n-2})\\
&\qquad\cdots(\tau_1\cdots\tau_{n-t})(\tau_1\cdots\tau_{l-(t+1)})\\
&\equiv \tau_{l-(t+1)} x_{l-(t+1)}\bigl(\prod_{j=2}^{t+1}x_j^{j-1}\bigr)\bigl(\prod_{j=t+2}^{l-1}x_j^{t+1}\bigr)\bigl(\prod_{j=l}^{n}x_j^{t}\bigr)x_n^{k-1}(\tau_1\cdots \tau_{n-1})(\tau_1\cdots \tau_{n-2})\\
&\qquad\cdots(\tau_1\cdots\tau_{n-t})(\tau_1\cdots\tau_{l-t-2})\\
&\equiv x_{l-(t+1)}\bigl(\prod_{j=2}^{t+1}x_j^{j-1}\bigr)\bigl(\prod_{j=t+2}^{l-1}x_j^{t+1}\bigr)\bigl(\prod_{j=l}^{n}x_j^{t}\bigr)x_n^{k-1}\tau_{l-(t+1)}(\tau_1\cdots \tau_{n-1})(\tau_1\cdots \tau_{n-2})\\
&\qquad\cdots(\tau_1\cdots\tau_{n-t})(\tau_1\cdots\tau_{l-t-2})\\
&\equiv x_{l-(t+1)}\bigl(\prod_{j=2}^{t+1}x_j^{j-1}\bigr)\bigl(\prod_{j=t+2}^{l-1}x_j^{t+1}\bigr)\bigl(\prod_{j=l}^{n}x_j^{t}\bigr)x_n^{k-1}(\tau_1\cdots \tau_{n-1})\cdots(\tau_1\cdots \tau_{n-l+t+2})\\
&\qquad\tau_1(\tau_1\cdots \tau_{n-l+t+3})\cdots(\tau_1\cdots\tau_{n-t})(\tau_1\cdots\tau_{l-t-2})\\
&\equiv 0\pmod{[\NH_n(\Z),\NH_n(\Z)]},
\end{align*}
where we have used the fact that $\tau_{l-(t+1)}$ commutes with $(x_{l-(t+1)}x_{l-t})^{l-t-1}$ in the third equality above. This proves (\ref{a2}) in this case.

Now we assume that $l>2(t+1)$. For any $a\geq 2$ satisfying $a(t+1)<l\leq n$, we can apply Lemma \ref{xlb} to get that for any $1\leq b\leq a$, $$
A_2\equiv x_{l-b(t+1)}X_{k,t,l}\tau_{l-(t+1)}\pmod{[\NH_n(\Z),\NH_n(\Z)]}.
$$
We choose $a\geq 2$ be such that $1\leq l-a(t+1)<t+1$. We can compute \begin{align*}
A_2&\equiv \bigl(x_{l-a(t+1)}\prod_{j=2}^{t+1}x_j^{j-1}\bigr)\bigl(\prod_{j=t+2}^{l-1}x_j^{t+1}\bigr)\bigl(\prod_{j=l}^{n}x_j^{t}\bigr)x_n^{k-1}(\tau_1\cdots \tau_{n-1})(\tau_1\cdots \tau_{n-2})\cdots\\
&\qquad (\tau_1\cdots \tau_{n-t})(\tau_1\cdots \tau_{l-t-2}\tau_{l-(t+1)})\\
&\equiv \tau_{l-(t+1)}\bigl(x_{l-a(t+1)}\prod_{j=2}^{t+1}x_j^{j-1}\bigr)\bigl(\prod_{j=t+2}^{l-1}x_j^{t+1}\bigr)\bigl(\prod_{j=l}^{n}x_j^{t}\bigr)x_n^{k-1}(\tau_1\cdots \tau_{n-1})(\tau_1\cdots \tau_{n-2})\cdots\\
&\qquad (\tau_1\cdots \tau_{n-t})(\tau_1\cdots \tau_{l-t-2})\pmod{[\NH_n(\Z),\NH_n(\Z)]}.
\end{align*}
Note that since $a>1$, $t+2\leq l-(t+1)<l-1$ and $\tau_{l-(t+1)}$ commutes with $\prod_{j=t+2}^{l-1}x_j^{t+1}$. Therefore, we have



\begin{align*}
A_2&\equiv\bigl(x_{l-a(t+1)}\prod_{j=2}^{t+1}x_j^{j-1}\bigr)\bigl(\prod_{j=t+2}^{l-1}x_j^{t+1}\bigr)\bigl(\prod_{j=l}^{n}x_j^{t}\bigr)x_n^{k-1}(\tau_1\cdots \tau_{n-1})(\tau_1\cdots \tau_{n-2})\cdots\\
&\qquad (\tau_1\cdots \tau_{n-t})(\tau_1\cdots \tau_{l-t-2})\tau_{l-2(t+1)}\\
&\equiv \tau_{l-2(t+1)}\bigl(x_{l-a(t+1)}\prod_{j=2}^{t+1}x_j^{j-1}\bigr)\bigl(\prod_{j=t+2}^{l-1}x_j^{t+1}\bigr)\bigl(\prod_{j=l}^{n}x_j^{t}\bigr)x_n^{k-1}(\tau_1\cdots \tau_{n-1})(\tau_1\cdots \tau_{n-2})\cdots\\
&\qquad (\tau_1\cdots \tau_{n-t})(\tau_1\cdots \tau_{l-t-2})\\
&\quad\vdots\\
&\equiv \tau_{l-a(t+1)}\bigl(x_{l-a(t+1)}\prod_{j=2}^{t+1}x_j^{j-1}\bigr)\bigl(\prod_{j=t+2}^{l-1}x_j^{t+1}\bigr)\bigl(\prod_{j=l}^{n}x_j^{t}\bigr)x_n^{k-1}(\tau_1\cdots \tau_{n-1})(\tau_1\cdots \tau_{n-2})\cdots\\
&\qquad (\tau_1\cdots \tau_{n-t})(\tau_1\cdots \tau_{l-t-2})\\
&\equiv \bigl(x_{l-a(t+1)}\prod_{j=2}^{t+1}x_j^{j-1}\bigr)\bigl(\prod_{j=t+2}^{l-1}x_j^{t+1}\bigr)\bigl(\prod_{j=l}^{n}x_j^{t}\bigr)x_n^{k-1}(\tau_1\cdots \tau_{n-1})\cdots(\tau_1\cdots \tau_{n-l+a(t+1)+1})\\
&\qquad \cdots\tau_1(\tau_1\cdots \tau_{n-l+a(t+1)})(\tau_1\cdots \tau_{n-t})(\tau_1\cdots \tau_{l-t-2})\equiv 0\pmod{[\NH_n(\Z),\NH_n(\Z)]},
\end{align*}
where in the last equality above we have used the fact that $\tau_{l-a(t+1)}$ commutes with $\bigl(x_{l-a(t+1)}\prod_{j=2}^{t+1}x_j^{j-1}\bigr)$ as $1\leq l-a(t+1)<t+1$.
This completes the proof of (\ref{a2}) and hence the proof of the (\ref{induction1}).

Finally, we take $t=n-2$ and $l=t+2=n$ in \eqref{induction1} and doing the same computation as in Case 1: \begin{align*}
X_{k,n}&\equiv X_{k,n-2,n}\equiv \bigl(\prod_{j=2}^{n-1}x_j^{j-1}\bigr)x_n^{n+k-3}(\tau_1\cdots \tau_{n-1})(\tau_1\cdots \tau_{n-2})\cdots(\tau_1\tau_2)\\
&=\bigl(\prod_{j=2}^{n-1}x_j^{j-1}\bigr)x_n^{n+k-3}(\tau_1\cdots \tau_{n-1})(\tau_1\cdots \tau_{n-2})\cdots(\tau_1\tau_2)(x_2\tau_1-\tau_1x_1)\\
&=\bigl(\prod_{j=2}^{n-1}x_j^{j-1}\bigr)x_n^{n+k-3}(\tau_1\cdots \tau_{n-1})(\tau_1\cdots \tau_{n-2})\cdots(\tau_1\tau_2x_2\tau_1)-\\
&\qquad \tau_1\bigl(x_1\prod_{j=2}^{n-1}x_j^{j-1}\bigr)x_n^{n+k-3}(\tau_1\cdots \tau_{n-1})(\tau_1\cdots \tau_{n-2})\cdots(\tau_1\tau_2)\\
&=\bigl(\prod_{j=2}^{n-1}x_j^{j-1}\bigr)x_n^{n+k-3}(\tau_1\cdots \tau_{n-1})(\tau_1\cdots \tau_{n-2})\cdots(\tau_1x_3\tau_2\tau_1)-\\
&\qquad \bigl(x_1\prod_{j=2}^{n-1}x_j^{j-1}\bigr)x_n^{n+k-3}\tau_1(\tau_1\cdots \tau_{n-1})(\tau_1\cdots \tau_{n-2})\cdots(\tau_1\tau_2)\\
&=\bigl(\prod_{j=2}^{n-1}x_j^{j-1}\bigr)x_n^{n+k-3}(\tau_1\cdots \tau_{n-1})(\tau_1\cdots \tau_{n-2})\cdots (\tau_1\tau_2\tau_3x_3)(\tau_1\tau_2\tau_1)\\
&=\bigl(\prod_{j=2}^{n-1}x_j^{j-1}\bigr)x_n^{n+k-3}(\tau_1\cdots \tau_{n-1})(\tau_1\cdots \tau_{n-2})\cdots (\tau_1\tau_2x_4\tau_3)(\tau_2\tau_1\tau_2)\\
&=\bigl(\prod_{j=2}^{n-1}x_j^{j-1}\bigr)x_n^{n+k-3}(\tau_1\cdots \tau_{n-1})(\tau_1\cdots \tau_{n-2})\cdots (\tau_1\tau_2\tau_3\tau_4x_4)(\tau_1\tau_2\tau_3)(\tau_2\tau_1\tau_2)\\
&=\bigl(\prod_{j=2}^{n-1}x_j^{j-1}\bigr)x_n^{n+k-3}(\tau_1\cdots \tau_{n-1})(\tau_1\cdots \tau_{n-2})\cdots (\tau_1\tau_2\tau_3x_5\tau_4)(\tau_1\tau_2\tau_3)(\tau_2\tau_1\tau_2)\\
&=\bigl(\prod_{j=2}^{n-1}x_j^{j-1}\bigr)x_n^{n+k-2}(\tau_1\cdots \tau_{n-1})(\tau_1\cdots \tau_{n-2})\cdots (\tau_1\tau_2\tau_3\tau_4)(\tau_1\tau_2\tau_3)(\tau_1\tau_2\tau_1)\\
&\quad\vdots\\
&=(\tau_1\cdots \tau_{n-1})(\tau_1\cdots \tau_{n-2})\cdots (\tau_1\tau_2\tau_3\tau_4)(\tau_1\tau_2\tau_3)(\tau_1\tau_2\tau_1)\bigl(\prod_{j=2}^{n-1}x_j^{j-1}\bigr)x_n^{n+k-2}\\
&=e_{[1,n]}x_n^{k-1}\pmod{[\NH_n(\Z),\NH_n(\Z)]}.
\end{align*}
\end{proof}

Recall that for any integers $1\leq a<b\leq n$, $w_{[a,b]}$ is defined to be the unique longest element in $\Sym_{\{a,a+1,\cdots,b-1,b\}}$, and
$$
e_{[a,b]}:=\tau_{w[a,b]}x_{a+1}x_{a+2}^2\cdots x_b^{b-a}\in \NH_n(\Z),\,\,w_{[a,a]}:=1 .
$$
Note that $e_{[a,b]}$ is a degree $0$ homogeneous idempotent. The following lemma can be regarded as a refinement and generalization of \cite[Lemma 3.17]{HS3}.

\begin{lem}\label{nil-Hecke relation}
Let $\O$ be an integral domain. For any $n\in\Z_{\geq 1}$ and $k\in\Z_{\geq 0}$, let $Z^{(k)}_n:=x_1^k\tau_1\tau_2\cdots \tau_{n-1}\in\NH_n(\O)$.  \begin{enumerate}
\item[(1)] If $n\geq 2$ and $0\leq k < n-1$, then we have $$ Z^{(k)}_n\in [\NH_n(\O),\NH_n(\O)];
$$
\item[(2)] If $k\geq n-1$, then we have $$
Z^{(k)}_n\in \sum_{d\in\Z_{\geq 1}}\sum_{\substack{l_j\geq 0, \sum_{j=1}^d l_j=k-(n-1)\\
1=a_1<a_2<\cdots<a_d<a_{d+1}=n+1}}\O \prod_{j=1}^d (e_{[a_j,a_{j+1}-1]}x_{a_{j+1}-1}^{l_j}) + [\NH_n(\O),\NH_n(\O)];
$$
\item[(3)] Let $d\in\Z_{\geq 1}$. For any integers $1=a_1<a_2<\cdots<a_d<a_{d+1}=n+1$, we have $$\prod_{j=1}^d e_{[a_j,a_{j+1}-1]}\in \O e_{[1,n]}+ [\NH_n(\O),\NH_n(\O)].$$ In particular, if $k=n-1$, then we have $$
Z^{(k)}_n\in \O e_{[1,n]}+ [\NH_n(\O),\NH_n(\O)].$$
\end{enumerate}
\end{lem}

\begin{proof}
By \cite[Proposition 2.21]{Rou2}, $\NH_n(\O)$ is a matrix algebra over the ring $\rm{Sym}_n$ of symmetric polynomials on $x_1,\cdots,x_n$. Let $\mathcal{B}({\rm{Sym}_n})$ be an integral basis of $\rm{Sym}_n$. It follows that $$
\bigl\{E_{ij}(f),E_{kk}(g)-E_{11}(g)\bigm|1\leq i\neq j\leq n!,2\leq k\leq n!, f,g\in\mathcal{B}({\rm{Sym}_n})\bigr\}
$$ gives an integral basis for the commutator subspace $[\NH_n(\O),\NH_n(\O)]$, which implies that the canonical map $\O\otimes_\Z [\NH_n(\Z),\NH_n(\Z)]\rightarrow [\NH_n(\O),\NH_n(\O)]$ is an isomorphism. Hence it suffices to prove this lemma for $\O=\Z$.

(1) Since
$\Tr(\NH_n(\Z)):=\NH_n(\Z)/[\NH_n(\Z),\NH_n(\Z)]$ is a free $\Z$-module with a basis $$
\bigl\{E_{11}(f)+[\NH_n(\Z),\NH_n(\Z)]\bigm|f\in\mathcal{B}({\rm{Sym}_n})\bigr\}.
$$
It follows that $[\NH_n(\Z),\NH_n(\Z)]$ is a pure $\Z$-submodule of $\NH_n(\Z)$. Hence we can deduce that $[\NH_n(\Z),\NH_n(\Z)]=\NH_n(\Z)\cap [\NH_n(\Q),\NH_n(\Q)]$. However, we have $$Z^{(k)}_n\in [\NH_n(\Q),\NH_n(\Q)]
$$ by \cite[Corollary 3.20]{HS3}.
It follows that $$
Z^{(k)}_n\in\NH_n(\Z)\cap [\NH_n(\Q),\NH_n(\Q)]=[\NH_n(\Z),\NH_n(\Z)].$$

(2) We use induction on $n$. The statement is trivial for $n=1$. Henceforth we assume $n>1$. We claim that for any  $2\leq v\leq n$,  \begin{equation}\label{induction3}
\begin{matrix}Z^{(k)}_n \in \sum_{d\in\Z_{\geq 1}}\sum_{\substack{l_j\geq 0, \sum_{j=1}^d l_j=k-(n-1)\\
1=a_1<a_2<\cdots<a_d<a_{d+1}=n+1}}\Z \prod_{j=1}^d e_{[a_j,a_{j+1}-1]}x_{a_{j+1}-1}^{l_j}+[\NH_n(\Z),\NH_n(\Z)]\\
+\biggl(\prod_{j=2}^{v-1}x_j\biggr)x_{v}^{k-(v-2)}\tau_1\tau_2\cdots\tau_{n-1}.\end{matrix}
\end{equation}

In fact, we have in $\Tr(\NH_n(\Z))$ that
$$\begin{aligned}
Z^{(k)}_n&\equiv \tau_{1}\tau_{2}\cdots \tau_{n-1} x_{1}^k \\
&\equiv x_{2}^k\tau_{1}\tau_{2}\cdots \tau_{n-1}-\sum_{\substack{k_1+k_2=k-1,\\ k_1,k_2\in\N}}x_1^{k_1}x_2^{k_2}\tau_2\cdots\tau_{n-1}\pmod{[\NH_n(\Z),\NH_n(\Z)]}.
\end{aligned}$$
Recall that the natural algebra embedding $\NH_{n-1}(\Z)\hookrightarrow \NH_n(\Z)$ defined on generators by $x_i\mapsto x_{i+1}$ and $\tau_j\mapsto \tau_{j+1}$ for $1\leq i<n,\,1\leq j<n-1$ induces a linear map:
$\Tr(\NH_{n-1}(\Z))\rightarrow \Tr(\NH_n(\Z))$. Applying induction hypothesis and the result (1), we get that $$\begin{aligned}
&\sum_{\substack{k_1+k_2=k-1,\\ k_1,k_2\in\N}}x_1^{k_1}x_2^{k_2}\tau_2\cdots\tau_{n-1}\\
&\qquad \in\sum_{\substack{k_1+k_2=k-1,\\ k_1,k_2\in\N}}  x_1^{k_1}\sum_{d\in\Z_{\geq 1}}\sum_{\substack{l_j\geq 0, \sum_{j=2}^d l_j=k_2-(n-2)\\
2=a_2<a_3<\cdots<a_d<a_{d+1}=n}}\Z \prod_{j=1}^d e_{[a_j,a_{j+1}-1]}x_{a_{j+1}-1}^{l_j}+ [\NH_n(\Z),\NH_n(\Z)]\\
&\qquad \subseteq \sum_{d\in\Z_{\geq 1}}\sum_{\substack{l_j\geq 0, \sum_{j=1}^d l_j=k-(n-1)\\
1=a_1<a_2<\cdots<a_d<a_{d+1}=n+1}}\Z \prod_{j=1}^d e_{[a_j,a_{j+1}-1]}x_{a_{j+1}-1}^{l_j}+ [\NH_n(\Z),\NH_n(\Z)],
\end{aligned}$$ where the last inclusion follows from $k_1+k_2-(n-2)=k-1-(n-2)=k-(n-1)$ when $k_1+k_2=k-1$. Hence the claim \eqref{induction3} is true for $v=2$.

Suppose the claim \eqref{induction3} is true for $2\leq v<n$. We only need to show that $$\begin{aligned}
&\biggl(\prod_{j=2}^{v-1}x_j\biggr)x_{v}^{k-(v-2)}\tau_1\tau_2\cdots\tau_{n-1}\in \bigl(\prod_{j=2}^{v}x_j\bigr)x_{v+1}^{k-(v-1)}\tau_{1}\tau_{2}\cdots \tau_{n-1}+ \\
&\qquad\qquad \sum_{d\in\Z_{\geq 1}}\sum_{\substack{l_j\geq 0, \sum_{j=1}^d l_j=k-(n-1)\\ 1=a_1<a_2<\cdots<a_d<a_{d+1}=n+1}}\Z \prod_{j=1}^d e_{[a_j,a_{j+1}-1]}x_{a_{j+1}-1}^{l_j}+[\NH_n(\Z),\NH_n(\Z)].\end{aligned}
$$
Since $v\leq n-1\leq k$, $k-(v-2)\geq 2$, we can compute $$\begin{aligned}
&\quad\, \bigl(\prod_{j=2}^{v-1}x_j\bigr)x_v^{k-(v-2)}\tau_{1}\tau_{2}\cdots \tau_{n-1}=x_v^{k-(v-1)}\bigl(\prod_{j=2}^{v}x_j\bigr)\tau_{1}\tau_{2}\cdots \tau_{n-1}\\
&\equiv \bigl(\prod_{j=2}^{v}x_j\bigr)\tau_{1}\tau_{2}\cdots \tau_{n-1}x_v^{k-(v-1)}\\
&\equiv\bigl(\prod_{j=2}^{v}x_j\bigr)x_{v+1}^{k-(v-1)}\tau_{1}\tau_{2}\cdots \tau_{n-1} -\sum_{k'_1+k'_2=k-v}  \bigl(\prod_{j=2}^{v}x_j\bigr)\tau_1\cdots \tau_{v-1}x_v^{k'_1}x_{v+1}^{k'_2}\tau_{v+1}\cdots\tau_{n-1}\\
&\equiv\bigl(\prod_{j=2}^{v}x_j\bigr)x_{v+1}^{k-(v-1)}\tau_{1}\tau_{2}\cdots \tau_{n-1} -\sum_{k'_1+k'_2=k-v}  \bigl(\prod_{j=2}^{v}x_j\bigr)\tau_1\cdots \tau_{v-1}x_{v+1}^{k'_2}\tau_{v+1}\cdots\tau_{n-1}x_v^{k'_1}\\
& \equiv \bigl(\prod_{j=2}^{v}x_j\bigr)x_{v+1}^{k-(v-1)}\tau_{1}\tau_{2}\cdots \tau_{n-1} -\sum_{k'_1+k'_2=k-v}  \bigl(\prod_{j=2}^{v-1}x_j\bigr) x_v^{1+k'_1} \tau_1\cdots \tau_{v-1}x_{v+1}^{k'_2}\tau_{v+1}\cdots\tau_{n-1}\\ &\quad \pmod{[\NH_n(\Z),\NH_n(\Z)]}.
\end{aligned}$$
We have a linear map $\Tr(\NH_v(\Z))\otimes\Tr(\NH_{n-v}(\Z))\rightarrow \Tr(\NH_n(\Z))$ which is induced by the natural algebra homomorphism $\NH_v(\Z)\otimes \NH_{n-v}(\Z)\rightarrow \NH_n(\Z)$ given by  $$
 x_{i'}\mapsto x_{i'},\,\,\tau_{j'}\mapsto\tau_{j'},\,\,x_i\mapsto x_{i+v},\,\,\tau_j\mapsto \tau_{j+v},\quad \forall\, 1\leq i'\leq v, 1\leq j'<v,\,1\leq i\leq n-v,\,1\leq j<n-v.
$$

Now apply (1), Lemma \ref{pre nil-Hecke relation} and  induction hypothesis to the second term in last paragraph, we have $$\begin{aligned}
&\quad\,\sum_{\substack{k'_1+k'_2=k-v,\\ k'_1,k'_2\in\N}}  \bigl(\prod_{j=2}^{v-1}x_j\bigr)x_v^{1+k'_1} \tau_1\cdots \tau_{v-1}x_{v+1}^{k'_2}\tau_{v+1}\cdots\tau_{n-1}\\
&\qquad \in \sum_{\substack{k'_1+k'_2=k-v,\\ k'_1,k'_2\in\N}} e_v x_v^{k'_1}\sum_{d\in\Z_{\geq 1}}\sum_{\substack{l_j\geq 0, \sum_{j=1}^d l_j=k'_2-(n-v-1)\\
1=a_1<a_2<\cdots<a_d<a_{d+1}=n-v+1}}\Z \prod_{j=1}^d e_{[a_j+v,a_{j+1}-1+v]}x_{a_{j+1}-1+v}^{l_j}\\
&\qquad\qquad + [\NH_n(\Z),\NH_n(\Z)]\\
&\qquad\subset\sum_{d\in\Z_{\geq 1}} \sum_{\substack{l_j\geq 0, \sum_{j=1}^d l_j=k-(n-1)\\
1=a_1<a_2<\cdots<a_d<a_{d+1}=n+1}}\Z \prod_{j=1}^d e_{[a_j,a_{j+1}-1]}x_{a_{j+1}-1}^{l_j}+ [\NH_n(\Z),\NH_n(\Z)],
\end{aligned}$$ the inclusion follows from $k'_1+k'_2-(n-v-1)=k-v-(n-v-1)=k-(n-1)$ since $k'_1+k'_2=k-v. $ This completes the proof of our claim \eqref{induction3}.

Finally, we take $v=n$ in the claim \eqref{induction3} and get $$\begin{aligned}
Z^{(k)}_n &\in \bigl(\prod_{j=2}^{n-1}x_j\bigr)x_n^{k-n+2}\tau_1\tau_2\cdots\tau_{n-1} + \sum_{d\in\Z_{\geq 1}}\sum_{\substack{l_j\geq 0, \sum_{j=1}^d l_j=k-(n-1)\\
1=a_1<a_2<\cdots<a_d<a_{d+1}=n+1}}\Z \prod_{j=1}^d e_{[a_j,a_{j+1}-1]}x_{a_{j+1}-1}^{l_j}\\
&\qquad + [\NH_n(\Z),\NH_n(\Z)].
\end{aligned}$$

However, by Lemma \ref{pre nil-Hecke relation}, $$\begin{aligned}
&\quad\,\bigl(\prod_{j=2}^{n-1}x_j\bigr)x_n^{k-n+2}\tau_1\tau_2\cdots\tau_{n-1}\\
&\in e_nx_n^{k-(n-1)}+[\NH_n(\Z),\NH_n(\Z)]\\
& \subseteq \sum_{d\in\Z_{\geq 1}}\sum_{\substack{l_j\geq 0, \sum_{j=1}^d l_j=k-(n-1)\\
1=a_1<a_2<\cdots<a_d<a_{d+1}=n+1}}\Z \prod_{j=1}^d e_{[a_j,a_{j+1}-1]}x_{a_{j+1}-1}^{l_j}+ [\NH_n(\Z),\NH_n(\Z)].
\end{aligned}$$ This completes the proof of (2).

(3) Note that $\prod_{j=1}^d e_{[a_j,a_{j+1}-1]}$ is an idempotent. By \cite[Lemma 2.19, Proposition 2.22]{Rou2}, $e_{[1,n]}$ is a primitive idempotent in $\NH_n(\Z)$ and $\NH_n(\Z)\simeq {\NH_n(\Z)e_{[1,n]}}^{\oplus n!}$ as left $\NH_n(\Z)$-module.
Since $\NH_n(\Z)$ has a unique indecomposable finitely generated projective module, the image of any two primitive idempotents in the cocenter are equal by \cite[\S6, Exercise 14]{CR}. Combing this with (2), we prove the lemma.
\end{proof}

A sequence of non-negative integers $\mathbf{a}=(a_1,a_2,\cdots,a_k)$ is called a composition of $n$ if $\sum_{i=1}^{k}a_i=n$. If $\mathbf{a}=(a_1,a_2,\cdots,a_k)$ is a composition of $n$ then we write $\mathbf{a}\models n$.
Let $\alpha\in Q_n^+$ and $\nu=(\nu_1,\cdots,\nu_n)\in I^\alpha$. Following \cite{HS3}, we define $$
\mathcal{C}(\nu):=\Biggl\{\mathbf{b}:=(b_1,\cdots,b_m)\models n\Biggm|\begin{matrix}\text{$m,b_1,\cdots,b_m\geq 1$, $\nu_j=\nu_{j+1}$, for any $1\leq i\leq m$}\\
\text{and any $\sum_{k=1}^{i-1}b_k+1\leq j\leq\sum_{k=1}^{i}b_k$}
\end{matrix}\Biggr\}.
$$
Let $\mathbf{b}=(b_1,\cdots,b_m)\in\mathcal{C}(\nu)$. Following \cite[\S3.1]{HS3}, we define $\mathbf{c}=(c_0,c_1,\cdots,c_{m})$ as follows: \begin{equation}\label{bc}
c_0:=0,\,\, c_j:=b_1+\cdots+b_j,\,\,j=1,2,\cdots,m.
\end{equation}

Let $K$ be an arbitrary field. We shall write $\NH_n:=\NH_n(K)$. Let $\alpha\in Q_n^+$ and $\R[\alpha]$ be the cyclotomic KLR algebra over $K$ which was recalled in the last section.

\begin{dfn}\text{(\cite[Definition 3.18]{HS3})}\label{ab} Let $1\leq m,m'\leq n$. We define $A_{m}$ to be the $K$-subalgebra of $\R[\alpha]$ generated by $$
\tau_w,\, x_j,\,\, w\in \Sym_{\{1,2,\cdots,m\}}, 1\leq j\leq m .
$$
We define $B_{m'-1}$ to be the $K$-subalgebra of $\R[\alpha]$ generated by $$
\tau_w,\, x_j,\,\, w\in \Sym_{\{m',m'+1,\cdots,n\}}, m'\leq j\leq n .
$$
We set $A_0=Ke(\alpha)=B_n$. We call $A_{m}$ the first $m$-th part of $\R[\alpha]$, while call $B_{m'-1}$ the last $(n-(m'-1))$-th part of $\R[\alpha]$.
\end{dfn}

The following lemma is a generalization of \cite[Corollary 3.20]{HS3} to ground field $K$ of arbitrary characteristic.

\begin{lem}\label{relations2} Let $K$ be an arbitrary field. Let $\nu\in I^\alpha$, $\mathbf{b}=(b_1,\cdots,b_m)\in\mathcal{C}(\nu)$ and $\mathbf{c}=(c_0,c_1,\cdots,c_m)$ be defined as in (\ref{bc}).
Suppose that $$y=y_1 x_{c_t+1}^k\tau_{c_t+1}\tau_{c_{t}+2}\cdots \tau_{c_{t+1}-1} y_2e(\nu),$$ where $k\in\N, y_1\in A_{c_t}e(\nu),\,y_2\in B_{c_{t+1}}e(\nu), 0\leq t\leq m-1$. \begin{enumerate}
\item[(1)] If $k<b_{t+1}-1$, then $y\in [\R[\alpha],\R[\alpha]]$;
\item[(2)] If $k\geq b_{t+1}-1$, then we have $$
y\in \sum_{d\in\Z_{\geq 1}}\sum_{\substack{l_j\geq 0, \sum_{j=1}^d l_j=k-(b_{t+1}-1)\\
c_t+1=a_1<a_2<\cdots<a_d<a_{d+1}=c_{t+1}+1}}K y_1 \Bigl(\prod_{j=1}^d e_{[a_j,a_{j+1}-1]}x_{a_{j+1}-1}^{l_j}\Bigr) y_2 e(\nu)+[\R[\alpha],\R[\alpha]];
$$
\item[(3)] If $k=b_{t+1}-1$, we have $$
y\in K y_1 e_{[c_t+1,c_{t+1}]} y_2 e(\nu)+[\R[\alpha],\R[\alpha]].
$$
\item[(4)] Let $d\in\Z_{\geq 1}$. For any integers $c_t+1=a_1<a_2<\cdots<a_d<a_{d+1}=c_{t+1}+1$, we have $$
y_1\Bigl(\prod_{j=1}^d e_{[a_j,a_{j+1}-1]}\Bigr)y_2e(\nu)\in K y_1e_{[c_t+1,c_{t+1}]}y_2e(\nu)+ [\R[\alpha],\R[\alpha]].$$
\end{enumerate}
\end{lem}

\begin{proof} Consider the following algebraic maps: $$\NH_{b_{t+1}}\rightarrow \RR_{\alpha} \rightarrow \R[\alpha],
$$ where the first map is determined uniquely by $x_a\mapsto x_{a+c_{t}}e(\nu)$ and $\tau_b\mapsto \tau_{b+c_{t}}e(\nu)$ for $1\leq a\leq b_{t+1},\,1\leq b\leq b_{t+1}-1$, and the second map is the natural surjection. This map induces a $K$-map between cocenters: $$\Tr(\NH_{b_{t+1}})\rightarrow \Tr(\R[\alpha]).
$$ Now note that both $y_1, y_2$ commute with the image of $\NH_{b_{t+1}}$, the Lemma follows from Lemma \ref{nil-Hecke relation}.
\end{proof}

In the rest of this section, we shall recall some notations and a result in \cite{HS3} which will be used in the next subsection.

Let $\nu=(\nu_1,\cdots,\nu_n)\in I^\alpha$, $\mathbf{b}=(b_1,\cdots,b_m)\in\mathcal{C}(\nu)$ and $\mathbf{c}=(c_0,c_1,\cdots,c_m)$ be defined as in (\ref{bc}). We can decompose $\nu$ as follows:
\begin{equation}\label{nu1}
\nu=(\underbrace{\nu^{1},\nu^{1},\cdots,\nu^{1}}_{b_{1}\,copies},\cdots,\underbrace{\nu^{m},\nu^{m},\cdots,\nu^{m}}_{b_{m}\,copies}),
\end{equation}
where $m,b_1,\cdots,b_m\in\Z^{\geq 1}$ with $\sum_{i=1}^{m}b_i=n$, $\nu^1,\cdots,\nu^m\in I$. Note that it could happen that $\nu^{j}=\nu^{j+1}$ for some $1\leq j<m$. We define \begin{equation}\label{cLamnu}
\mathcal{C}^\Lam(\nu):=\biggl\{\mathbf{b}=(b_1,\cdots,b_m)\in\mathcal{C}(\nu)\biggm|\lam_{c_i,\nu}:=\<\alpha_{\nu^{i+1}},\Lam-\sum_{j=1}^{c_i}\alpha_{\nu_j}\> >0, \forall\,0\leq i<m\biggr\}.
\end{equation}

\begin{dfn} Let $\nu\in I^\alpha$. Let $\R[\nu,1]$ be the $K$-subspace of $\R[\alpha]$ spanned by the elements of the following form: \begin{equation}\label{pregenerator}
\Biggl\{\prod\limits_{0\leq i<m}\Bigl(x_{c_i+1}^{n_{\bc,i+1}}\tau_{c_i+1}\tau_{c_i+2}\cdots\tau_{c_{i+1}-1}\Bigr)e(\nu)\Biggm|
\begin{matrix}\text{$\mathbf{b}\in\mathcal{C}^\Lam(\nu), 0\leq n_{\bc,i+1}<\lam_{c_i,\nu}, \forall\,0\leq i<m$,}\\
\text{where $\{c_j|0\leq j\leq m\}$ is defined}\\ \text{using $\mathbf{b}$ as in (\ref{bc}).} \end{matrix}\Biggr\}.
\end{equation}
\end{dfn}

Note that in the above definition, any two factors in the product $\Bigl(x_{c_i+1}^{n_{\bc,i+1}}\tau_{c_i+1}\tau_{c_i+2}\cdots\tau_{c_{i+1}-1}\Bigr)e(\nu)$,  $\Bigl(x_{c_j+1}^{n_{\bc,j+1}}\tau_{c_j+1}\tau_{c_j+2}\cdots\tau_{c_{j+1}-1}\Bigr)e(\nu)$ actually commute with each other. So it makes no confusion without specifying the order of the product.

For any subset $A$ of $\R[\alpha]$, we use $\overline{A}$ to denote the natural image of $A$ in the cocenter $\R[\alpha]/[\R[\alpha],\R[\alpha]]$ of $\R[\alpha]$.

\begin{thm}\text{(\cite[Theorem 3.7]{HS3})}\label{generator} For any $\alpha\in Q_n^+$, we have $$
\Tr(\R[\alpha])=\RR^\Lam_\alpha/[\RR^\Lam_\alpha,\RR^\Lam_\alpha]=\sum_{\nu\in I^\alpha}\overline{\RR^\Lam_{\nu,1}}.$$
\end{thm}

We end this section with the following corollary.

\begin{cor}\label{2z} For any homogeneous element $0\neq y\in\Tr(\R[\alpha])$ and $0\neq z\in Z(\R[\alpha])$, both $\deg y$ and $\deg z$ are even integers.
\end{cor}

\begin{proof} This follows from Theorem \ref{generator} and the fact that $d_{\Lam,\alpha}$ is even.
\end{proof}

\bigskip

\section{Explicit generators for the cocenter over arbitrary field}

The maximal degree and the minimal degree of the cocenter $\Tr(\R[\alpha])$ of $\R[\alpha]$ were determined in \cite{SVV} and \cite{HS3}. In \cite[Section 4]{HS3}, the authors further constructed some explicit $K$-linear generators for both the maximal degree component and the minimal degree component of $\Tr(\R[\alpha])$ via the so-called ``piecewise dominant sequences'' introduced in \cite{HS3}. The latter results together with the maximal degree results are all presented under the assumption that the ground field has characteristic $0$. The purpose of this section is to remove the characteristic $0$ assumption on the ground field $K$. That says, we shall generalize those results in \cite[Section 4]{HS3} to the ground field $K$ of arbitrary characteristic.  Moreover, we shall characterize the dimension of the degree zero part of the cocenter $\Tr(\R[\alpha])$.

\subsection{Generators of the cocenter}

Let $\alpha\in Q_n^+, \Lam\in P^+$ and $\nu\in I^\alpha$. There is a unique decomposition of $\nu$ as follows: \begin{equation}\label{decomp4}
\nu=(\nu_1,\cdots,\nu_n)=(\underbrace{\nu^{1},\nu^{1},\cdots,\nu^{1}}_{b_{1}\,copies},\cdots,\underbrace{\nu^{m},\nu^{m},\cdots,\nu^{m}}_{b_{m}\,copies}),
\end{equation}
which satisfies that \begin{equation}\label{newrequire}
\nu^{j}\neq\nu^{j+1},\quad \forall\, 1\leq j<m,
\end{equation}
where $m,b_1,\cdots,b_m\in\Z^{\geq 1}$ with $\sum_{j=1}^{m}b_j=n$.

For each $1\leq j\leq m$, we define \begin{equation}\label{elli30}
\ell_j(\nu):=\<h_{\nu^j}, \Lam-\sum_{t=1}^{c_{j-1}}\alpha_{\nu_t}\>,
\end{equation}
where $c_0:=0, c_j:=\sum_{s=1}^{j}b_s, \forall\,1\leq j\leq m$.
When $\nu$ is clear from the context, we shall write $\ell_j$ instead of $\ell_j(\nu)$ for simplicity.

\begin{dfn}\text{(\cite[Definition 4.4]{HS3})}\label{pddfn} Let $\Lam\in P^+$ and $\alpha\in Q_n^+$. We call $\nu=(\nu_1,\cdots,\nu_n)\in I^\alpha$ a {\bf piecewise dominant sequence} with respect to $\Lam$, if for the unique decomposition (\ref{decomp4}) of $\nu$ and any $1\leq j\leq m$, \begin{equation}\label{elli}
\ell_j=\ell_j(\nu)\geq b_j.
\end{equation}
\end{dfn}


\begin{lem}\text{(\cite[Lemma 4.7]{HS3})}\label{princlple criterion} Let $\nu=(\nu_1,\cdots,\nu_n)\in I^\alpha$ and fix the unique decomposition (\ref{decomp4}) of $\nu$. Then $\nu$ is a piecewise dominant sequence with respect to $\Lam$ if and only if for each $1\leq i\leq m$, there is an integer $c_{i-1}+1\leq k'_i\leq c_{i}$ such that
\begin{equation}\label{pdCondition}
\<h_{\nu^i}, \Lam-\sum_{j=1}^{k'_i-1}\alpha_{\nu_j}\>\geq c_i-k'_i+1.\end{equation}
In this case, we denote the maximal value of each $k'_i$ by $k_i$, which can be taken as: \begin{equation}\label{maxki}
k_i=\begin{cases} c_i, &\text{if $\ell_i-2b_i\geq 0$;}\\
\ell_i+2c_{i-1}-c_{i}+1, &\text{if $\ell_i-2b_i\leq -1$.}
\end{cases}
\end{equation}
\end{lem}

\begin{lem}\label{refineresidue}
Let $K$ be an arbitrary field. Let $\nu\in I^\alpha$ and $z\in \RR^\Lam_{\nu,1}$. If $\nu$ is not piecewise dominant with respect to $\Lam$, then $z\in[\R[\alpha],\R[\alpha]]$.
\end{lem}

\begin{proof} This can be proved by using the same argument in the proof of \cite[Lemma 4.12]{HS3} with Corollary 3.20 there replaced by Lemma \ref{relations2} here.
\end{proof}

The following definition of $Z^\Lam (\nu)$ is a generalization of \cite[Definition 4.10]{HS3} to arbitrary ground field.

\begin{dfn}\label{sdfns} Let $\nu\in I^\alpha$ be a piecewise dominant sequence with the unique decomposition (\ref{decomp4}). For $1\leq t\leq m$, we define the following elements and subsets of
elements in $\R[\alpha]$:
$$
Z^\Lam(\nu)_t:=\begin{cases}
x_{c_{t-1}+1}^{\ell_t-1}x_{c_{t-1}+2}^{\ell_t-3}\cdots  x_{c_t}^{\ell_t-2b_t+1}e(\nu), &\text{if $\ell_t\geq 2b_t-1$;}\\
x_{c_{t-1}+1}^{\ell_t-1}
x_{c_{t-1}+2}^{\ell_t-3}\cdots  x_{\ell_t+2c_{t-1}-c_t}^{2b_t-\ell_t+1} e_{[\ell_t+2c_{t-1}-c_t+1,c_{t}]}e(\nu)&\text{if $b_t<\ell_t < 2b_t-1$;}\\
e_{[c_t+1,c_{t+1}]}e(\nu)\, &\text{if $\ell_t=b_t$,}
\end{cases}
$$
\footnote{In \cite[Definition 4.10]{HS3}, $t=2b_t-1$ is included in the second case. However, it's not hard to check this can be also included in the first case.}and  $$
\mR^\Lam(\nu)_{t}:=\Biggl\{\prod_{j=1}^d e_{[a_j,a_{j+1}-1]}x^{l_j}_{a_{j+1}-1}e(\nu)\Biggm|\begin{matrix}0\leq l_j<\ell_t,\,\forall\,1\leq j\leq d, \\
c_t+1=a_1<a_2<\cdots<a_d<a_{d+1}=c_{t+1}+1
\end{matrix}\Biggr\}, $$
Furthermore, we set $$
Z^\Lam (\nu):=\prod_{t=1}^{m} Z^\Lam(\nu)_{t},\,\,\,  \mR^\Lam (\nu):=\prod_{t=1}^{m}\mR^\Lam(\nu)_{t},
\,\,\,e(\nu)^{(-)}:=\prod_{t=1}^{m} e_{[c_t+1,c_{t+1}]}e(\nu).
$$
\end{dfn}

Note that by construction, $\nu_{c_t+1}=\nu_{c_t+2}=\cdots=\nu_{c_{t+1}}$, it follows that each $e_{[c_t+1,c_{j+1}]}$ commutes with $e(\nu)$ in the above definition of $e(\nu)^{(-)}$. In particular, the idempotent $e(\nu)^{(-)}$ can be regarded as a ``divided power'' form of $e(\nu)$.

%

\begin{lem}\label{degree1} Let $K$ be an arbitrary field. Suppose $\nu\in I^\alpha$ is a piecewise dominant sequence with respect to $\Lam$, then $\deg(Z^\Lam (\nu))=d_{\Lam,\alpha}$.
\end{lem}

\begin{proof} Note that those idempotents $e_{[u,v]}$ are all homogeneous with degree $0$. The lemma follows directly from \cite[Lemma 4.11]{HS3}.
\end{proof}


%
%
%
%

\medskip
\noindent
\textbf{Proof of Theorem \ref{mainthm2}:} Part (2) is a consequence of Part (3) and Lemma \ref{relations2}. It remains to show Part (1) and Part (3).

Replacing \cite[Lemma 4.12]{HS3} (which is valid over characteristic $0$ field) with Lemma \ref{refineresidue} (which is valid over arbitrary ground field) in the argument, one can prove Part (3) by the same argument as that used in the proof of Part (3) of \cite[Theorem 1.3]{HS3},

It remains to prove Part (2). By Theorem \ref{generator} and Lemma \ref{refineresidue}, we can obtain a set of spanning elements of $\Tr(\R[\alpha])$ which are the images of some elements of the form \begin{equation}\label{range1}
\prod_{0\leq i<m'}\Bigl(x_{c'_i+1}^{n_{\bc',i+1}}\tau_{c'_i+1}\tau_{c'_i+2}\cdots\tau_{c'_{i+1}-1}\Bigr)e(\nu),
\end{equation}
where $\mathbf{b}'=(b'_1,\cdots,b'_{m'})\in\mathcal{C}^\Lam(\nu), 0\leq n_{\bc',i+1}<\lam_{c'_i,\nu}, c'_i:=\sum_{j=1}^{i}b'_j, \forall\,0\leq i<m'$, and $\nu$ is piecewise dominant with respect to $\Lam$.

Now let $\nu\in I^\alpha$ be a piecewise dominant sequence with respect to $\Lam$ with the unique decomposition (\ref{decomp4}). Following the proof of \cite[Theorem 1.3]{HS3}, we define $$
Z'(\nu)=Z'(\nu)_1Z'(\nu)_2\cdots Z'(\nu)_m, $$
where for each $1\leq i\leq m$, $$
Z'(\nu)_i:=\begin{cases}
x_{c_{i-1}+1}^{\ell_i-1}
x_{c_{i-1}+2}^{\ell_i-3}\cdots  x_{c_i}^{\ell_i-2b_i+1}e(\nu), &\text{if $\ell_i\geq 2b_i$;}\\
x_{c_{i-1}+1}^{\ell_i-1}
x_{c_{i-1}+2}^{\ell_i-3}\cdots  x_{k_i}^{\ell_i-2(k_i-c_{i-1})+1}\tau_{k_i}\cdots\tau_{c_i-2}\tau_{c_i-1}e(\nu), &\text{if $\ell_i\leq 2b_i-1$,}
\end{cases}
$$
where $k_i$ is as defined in (\ref{maxki}).

By the same argument used in the proof of Part (1) of \cite[Theorem 1.3]{HS3}, in order to give a spanning set for the maximal degree component of $\Tr(\R[\alpha])$, we only need to consider the monomial of the form $Z'(\nu)$. Now applying Corollary \ref{relations2}(3), we see that $Z'(\nu)$ is some $\Z$ scalar multiple of $Z^\Lam(\nu)$ whose degree is exactly $d_{\Lam,\alpha}$ by Lemma \ref{degree1}. In particular, this means the maximal degree of $\Tr(\R[\alpha])$ is $d_{\Lam,\alpha}$. Combining these with the discussion in the last two paragraphs, we can deduce that the image of those $Z^\Lam(\nu)$ can form a spanning set for the degree $d_{\Lam,\alpha}$ component of the cocenter $\Tr(\R[\alpha])$. This completes the proof of Part (1) of the theorem.
\hfill\qed

\medskip

\begin{rem}
Using Lemma \ref{idempn}, it is easy to see that $$
(v-u+1)!e_{[u,v]}e(\nu)-e(\nu)\in [\R[\alpha],\R[\alpha]]$$
whenever $\nu_{u}=\nu_{u+1}=\cdots=\nu_{v}$. From this equality one can immediately recover \cite[Theorem 1.3 (1), (2)]{HS3} from Theorem \ref{mainthm2} (1), (2).
\end{rem}

As a direct application of Theorem \ref{mainthm2}, we can deduce Theorem \ref{mainthm1}, which removes the characteristic $0$ assumption in \cite[Theorem 3.31(a)]{SVV} and \cite[Corollary 4.21]{HS3}.


\medskip
\noindent
\textbf{Proof of Theorem \ref{mainthm1}:} Suppose that $\bigl(\Tr(\R[\alpha])\bigr)_j\neq 0$. Then $j\geq 0$ by \cite[Proposition 3.14]{HS3} or Theorem \ref{mainthm2} (3). By the proof of Theorem \ref{mainthm2}, we see that $j\leq d_{\Lam,\alpha}$. Applying Corollary we get that $j\in 2\Z$. This proves the first part of the theorem. The second part of the theorem follows from Lemma \ref{tLam} and Corollary \ref{2z}.
\hfill\qed

\subsection{On the degree zero component of  $\Tr(\RR^\Lam_\alpha)$}

The purpose of this subsection is to use our main result Theorem \ref{mainthm2} to show that $\dim\Tr(\RR^\Lam_\alpha)\bigr)_{0}=\dim V(\Lam)_{\Lam-\alpha}$ holds for arbitrary ground field $K$. In particular, we shall give a $K$-basis of $\Tr(\RR^\Lam_\alpha)\bigr)_{0}$.

Recall that $\R[\alpha]\lmod$ is the category of finite dimensional (ungraded) $\R[\alpha]$-modules. We set $m_{\alpha,\Lam}:=\dim V(\Lam)_{\Lam-\alpha}$. Then $m_{\alpha,\Lam}=\#\Irr\R[\alpha]\lmod$ by \cite[Theorem 6.1]{KK}. Let $L_1,\cdots, L_{m_{\alpha,\Lam}}$ be the complete set of pairwise non-isomorphic ungraded irreducible modules in $\RR^\Lam_\alpha\lmod$. For each $1\leq j\leq m_{\alpha,\Lam}$, let $P_j$ be the project cover of $L_j$ in $\R[\alpha]\lmod$. Let $\circledast$ be the duality functor on $\R[\alpha]\lgmod$ which is induced by the anti-involution $*$ of $\R[\alpha]$. Since $\R[\alpha]$ is a $\Z$-graded Artin algebra, we know that both simple module $L_j$ and indecomposable module $P_j$ admit $\Z$-graded lifts (\cite{GG}).

\begin{lem}\text{(\cite[Lemma 3.5]{Br2})} For each $1\leq j\leq m_{\alpha,\Lam}$, there is a $\Z$-graded lift $\mathbb{L}_j$ of $L_j$ such that $\bigl(\mathbb{L}_j\bigr)^{\circledast}\cong\mathbb{L}_j$.
\end{lem}

For each $1\leq j\leq m_{\alpha,\Lam}$, we denote by $\mathbb{P}_j$ the graded projective cover of $\mathbb{L}_j$ in $\R[\alpha]\lgmod$. Then $\mathbb{P}_j$ is a $\Z$-graded lift of $P_j$. Let $q$ be an indeterminate over $\Z$.
For each $M\in \R[\alpha]\lgmod$, we define the graded dimension of $M$ as $\dim_q(M):=\sum_{k\in\Z}\dim (M_k)q^k$. The Grothendieck group of $\R[\alpha]\lgmod$ naturally becomes a $\Z[q,q^{-1}]$-module, where $q^k[M]:=[M\<k\>]$ for any $k\in\Z$ and $M\in \R[\alpha]\lgmod$.

\begin{cor}\label{Pjcor} In the Grothendieck group of $\R[\alpha]\lgmod$, we have \begin{equation}\label{Pdec1}
[\R[\alpha]]=\sum_{k=1}^{m_{\alpha,\Lam}}\dim_q(\mathbb{L}_k)[\mathbb{P}_k].
\end{equation}
Moreover, for each $1\leq j\leq m_{\alpha,\Lam}$, we have $\mathbb{P}_j^{\circledast}\cong\mathbb{P}_j\<-d_{\Lam,\alpha}\>$. In particular, $\mathbb{P}_j$ has a unique simple head $\mathbb{L}_j$ and a unique simple socle
$\mathbb{L}_j\<d_{\Lam,\alpha}\>$.
\end{cor}

\begin{proof} The equality (\ref{Pdec1}) follows from the equality $\Hom_{\R[\alpha]}(\R[\alpha],\mathbb{L}_k)\cong \mathbb{L}_k$. Applying Lemma \ref{tLam} we get that $$
q^{-d_{\Lam,\alpha}}[\R[\alpha]]=[(\R[\alpha])^{\circledast}]=\sum_{j=1}^{m_{\alpha,\Lam}}\dim_q(\mathbb{L}_j)[(\mathbb{P}_j)^{\circledast}].
$$
A priori, we know that $({P}_j)^{\circledast}$ is isomorphic to ${P}_j$ if we forget the $\Z$-grading. Thus $(\mathbb{P}_j)^{\circledast}\cong \mathbb{P}_j\<a_j\>$ for some $a_j\in\Z$. Note also that $\{[\mathbb{P}_j]|1\leq j\leq m_{\alpha,\Lam}\}$ are $\Z[q,q^{-1}]$-linearly independent. Combining this with the above equality we can deduce that $\mathbb{P}_j^{\circledast}\cong\mathbb{P}_j\<-d_{\Lam,\alpha}\>$ for each
$1\leq j\leq m_{\alpha,\Lam}$. In particular, $\mathbb{P}_j$ has a unique simple head $\mathbb{L}_j$ and a unique simple socle
$\mathbb{L}_j\<d_{\Lam,\alpha}\>$.
\end{proof}

The following result generalize \cite[Theorem 3.31(c)]{SVV} to arbitrary ground field.


\medskip
\noindent
\textbf{Proof of Theorem \ref{mainthm3}:} By \cite[\S6, Exercise 14]{CR}, if $e,f\in\R[\alpha]$ are two primitive idempotents such that $\R[\alpha]e\cong\R[\alpha]f$, then $f=aea^{-1}$ for some invertible element $a\in(\R[\alpha])^\times$, hence $f-e\in[\R[\alpha],\R[\alpha]]$.
Combining this with \cite[Proposition 5.8]{GG}, we can assume without loss of generality that for each $1\leq i\leq m_{\alpha,\Lam}$, $f_i$ is a degree $0$ homogeneous primitive idempotent in $\RR^\Lam_\alpha$ such that $\mathbb{P}_i\cong\RR^\Lam_\alpha f_i$. By Theorem \ref{mainthm2} (2), the degree zero component $\Tr(\R[\alpha])_0$ of $\Tr(\RR^\Lam_\alpha)$ is generated by $ae(\nu)+[\R[\alpha],\R[\alpha]]$, where $\nu$ is a piecewise dominant sequence with respect to $\Lam$ and $a$ is an idempotent with $ae(\nu)=e(\nu)a$. In particular, $ae(\nu)=e(\nu)a$ is an idempotent. It follows from  \cite[\S6, Exercise 14]{CR} that $ae(\nu)\in \sum_{t=1}^{m_{\alpha,\Lam}}a_tf_t+[\R[\alpha],\R[\alpha]]$ for some $a_t\in \N$, and we have $$\dim\, \bigl(\Tr(\RR^\Lam_\alpha)\bigr)_{0}\leq m_{\alpha,\Lam}.
$$
It remains to show that $\dim\Tr(\RR^\Lam_\alpha)_{0}\geq m_{\alpha,\Lam}$. Without loss of generality, we can assume that $K$ is an algebraically closed field by \cite[Proposition 2.1]{SVV}.

We define the basic algebra of $\R[\alpha]$ as follows:  $$
B^\Lam_\alpha:=\End_{\R[\alpha]}\bigl(\oplus_{i=1}^{m_{\alpha,\Lam}}\mathbb{P}_i\bigr).
$$
Then $B^\Lam_\alpha$ is $\Z$-graded and $Z(\RR^\Lam_\alpha)\cong Z(B^\Lam_\alpha)$ as $\Z$-graded algebra. In particular, $$
\dim \Tr(\RR^\Lam_\alpha)_{0}=\dim Z(\RR^\Lam_\alpha)_{d_{\Lam,\alpha}}=\dim Z(B^\Lam_\alpha)_{d_{\Lam,\alpha}} .
$$
It remains to show that $\dim Z(B^\Lam_\alpha)_{d_{\Lam,\alpha}}\geq m_{\alpha,\Lam}$.

Applying Corollary \ref{Pjcor}, $\head(\mathbb{P}_i)\cong\soc(\mathbb{P}_i)\<-d_{\Lam,\alpha}\>$ for each $1\leq i\leq m_{\alpha,\Lam}$. Therefore, we can fix an isomorphism $\psi_i: \head(\mathbb{P}_i)\cong\soc(\mathbb{P}_i)\<-d_{\Lam,\alpha}\>$, and a homomorphism $\theta_i\in B^\Lam_\alpha$ for each $1\leq i\leq m_{\alpha,\Lam}$, such that \begin{enumerate}
\item[1)] $\im(\theta_i)=\soc \mathbb{P}_i$; and
\item[2)] $\theta_i$ induces the isomorphism $\psi_i$;
\item[3)] $\theta_i(\mathbb{P}_j)=0$ for any $1\leq j\leq m_{\alpha,\Lam}$ with $j\neq i$.
\end{enumerate}

It is clearly that each $\theta_i$ is homogeneous of degree $d_{\Lam,\alpha}$, and the maps $\{\theta_i|1\leq i\leq m_{\alpha,\Lam}\}$ are $K$-linearly independent. It remains to show that $\theta_i\in Z(B^\Lam_\alpha)$ for each $1\leq i\leq m_{\alpha,\Lam}$. For any $\varphi\in\Hom_{\R[\alpha]}(\mathbb{P}_j,\mathbb{P}_k)$ with either $j\neq i$ or $k\neq i$, it is easy to see that $\varphi\circ\theta_i=0=\theta_i\circ\varphi$ (by considering them as elements in $B_\alpha$). Therefore, it suffices to show that for any $\psi\in\End_{\R[\alpha]}(\mathbb{P}_i)$, we have $\psi\circ\theta_i=\theta_i\circ\psi$.

Let $\psi\in\End_{\R[\alpha]}(\mathbb{P}_i)$. Suppose that $\psi$ is not injective. Then $\psi(\soc \mathbb{P}_i)=0$ and $\im(\psi)\subseteq\rad \mathbb{P}_i$. In this case, it is again easy to see that $\psi\circ\theta_i=0=\theta_i\circ\psi$ because
$\rad \mathbb{P}_i=\ker\theta_i$. Now assume that $\psi$ is injective and hence $\psi$ is an automorphism of $\mathbb{P}_i$. Recall that $\mathbb{L}_i=\head \mathbb{P}_i$. We fix an element $z_i\in \mathbb{P}_i\setminus\rad \mathbb{P}_i$. Then $\mathbb{P}_i=\R[\alpha]z_i$ because $\mathbb{L}_i$ is the unique simple head of $\mathbb{P}_i$. Note that the automorphism $\psi$ must induce an automorphism of $\head(\mathbb{P}_i)=\mathbb{L}_i$ and $\End_{\R[\alpha]}(\mathbb{L}_i)=K$ (because $K$ is algebraically closed). It follows that we can find a nonzero scalar $0\neq c\in K$ and an element $u_i\in\rad \mathbb{P}_i$ such that $$
\psi(z_i)=cz_i+u_i .
$$
Therefore, we can write $\psi=c\id_{\mathbb{P}_i}+\psi'$, where $\psi'$ is a well-defined endomorphism of $\mathbb{P}_i$ which is determined by $\psi'(az_i):=au_i, \forall\,a\in\R[\alpha]$, because both $\psi$ and $c\id_{\mathbb{P}_i}$ are well-defined.
Since $\psi'$ is not injective, we have $\psi'\circ\theta_i=0=\theta_i\circ\psi'$ by the discussion at the beginning part of this paragraph . Note that $c\id_{\mathbb{P}_i}\circ\theta_i=c\theta_i=\theta_i\circ c\id_{\mathbb{P}_i}$. It follows that $\psi\circ\theta_i=\theta_i\circ\psi$. This completes the proof of the theorem.
\hfill\qed\medskip

Let $\alpha\in Q_n^+$. Let $\mathscr{P\!D}_\alpha$ be the set of piecewise dominant sequences in $I^\alpha$ with respect to $\Lam$. We define $$
e_{\alpha,\Lam}:=\sum_{\nu\in\mathscr{P\!D}_\alpha}e(\nu),\,\, e_{\alpha,\Lam}^{(-)}:=\sum_{\nu\in\mathscr{P\!D}_\alpha}e(\nu)^{(-)},
$$

\begin{cor} There are Morita equivalences: $$
\R[\alpha] \overset{Morita}{\sim} e_{\alpha,\Lam}\R[\alpha]e_{\alpha,\Lam}\overset{Morita}{\sim} e_{\alpha,\Lam}^{(-)}\R[\alpha]e_{\alpha,\Lam}^{(-)}.
$$
\end{cor}

\begin{proof} By Theorem \ref{mainthm2} (2), the degree zero component $\Tr(\R[\alpha])_0$ of $\Tr(\RR^\Lam_\alpha)$ is generated by $e(\nu)^{(-)}+[\R[\alpha],\R[\alpha]]$, where $\nu$ is a piecewise dominant sequence with respect to $\Lam$. Suppose that there exists some primitive idempotent $f_i$ such that the ungraded indecomposable projective module $\R[\alpha]f_i$ does not occur as a direct summand of $\R[\alpha]e(\nu)^{(-)}$ for any piecewise dominant sequence $\nu$. Then we can deduce from the proof of Theorem \ref{mainthm3} that $\Tr(\R[\alpha])_0$ can be generated by a subset  $\{f_{j}+[\R[\alpha],\R[\alpha]]|1\leq j\leq m_{\alpha,\Lam},\,j\neq i\}$ with cardinality lesser than $m_{\alpha,\Lam}$, which contradicts to Theorem \ref{mainthm3}.
This proves $\R[\alpha] \overset{Morita}{\sim} e_{\alpha,\Lam}^{(-)}\R[\alpha]e_{\alpha,\Lam}^{(-)}$. As a result, we can also deduce that
$\R[\alpha] \overset{Morita}{\sim} e_{\alpha,\Lam}\R[\alpha]e_{\alpha,\Lam}$ because $\R[\alpha]e_{\alpha,\Lam}^{(-)}$ is a direct summand of $\R[\alpha]e_{\alpha,\Lam}$.
\end{proof}


\bigskip
\begin{thebibliography}{2}


\bibitem{BGHL} {\sc A.~Beliakova, Z.~Guliyev, K.~Habiro, and A.~Lauda}, {\em Trace as an alternative
decategorification functor}, Acta Math. Vietnam., {\bf 39} (2014), 425--480.

\bibitem{BHLZ} {\sc A.~Beliakova, K.~Habiro, A.~Lauda and M.~Zivkovic}, {\em Trace decategorification of
categorified quantum $\mathfrak{sl}_2$}, Math. Ann., {\bf 367} (2017), 397--440.

\bibitem{BHLW} {\sc A.~Beliakova, K.~Habiro, A.~Lauda and B.~Webster}, {\em Current algebras and categorified quantum groups}, J. London Math. Soc., {\bf 95}(2) (2017), 248--276.

\bibitem{Bow}
{\sc C.~Bowman}, {\em The many integral graded cellular bases of Hecke algebras of complex reflection groups}, Amer. J. Math., {\bf 144}(2) (2022), 437--504.



\bibitem{Br2}
{\sc J.~Brundan},  {\em Quiver Hecke algebras and categorification}, in: ``Advances in Representation Theory of Algebras,'' eds: D. Benson et al, EMS Congress Reports, 2013, pp.103--133.

\bibitem{BK:GradedKL}
{\sc J.~Brundan and A.~Kleshchev},  {\em Blocks of cyclotomic {H}ecke algebras and {K}hovanov-{L}auda algebras}, Invent. Math., {\bf 178}
  (2009), 451--484.

\bibitem{CR}
{\sc C.W.~Curtis and I.~Reiner}, {\em Methods of Representation theory, with applications to finite groups
  and orders}, Vol. I, A Wiley Interscience, 1981.

\bibitem{DVV}
{\sc O.~Dudas, M.~Varagnolo and E.~Vasserot}, {\em Categorical actions on unipotent representations of finite unitary groups}, Publ. Math. Inst. Hautes \'Etudes Sci., {\bf 129} (2019), 129--197.

\bibitem{Ev} {\sc A.~Evseev}, {\em RoCK blocks, wreath products and KLR algebras}, Math. Ann., {\bf 369} (2017), 1383--1433.

\bibitem{EK} {\sc A.~Evseev and A.~Kleshchev}, {\em Blocks of symmetric groups, semicuspidal KLR algebras and zigzag Schur-Weyl duality}, Ann. of Math., {\bf 188} (2018), 453--512.

\bibitem{GG}
{\sc R.~Gordon and E.L.~Green}, {\em Graded Artin algebras}, J. Algebra, {\bf 76} (1982), 111--137.


\bibitem{HM}
{\sc J.~Hu and A.~Mathas}, {\em Graded cellular bases for the cyclotomic Khovanov-Lauda-Rouquier algebras of type $A$}, Adv. Math., {\bf 225}(2) (2010), 598--642.


\bibitem{HS}
{\sc J.~Hu and L.~Shi}, {\em Graded dimensions and monomial bases for the cyclotomic quiver Hecke algebras}, Commun. Contemp. Math., (2023), article in press, doi:10.1142/S021919972350044X.



\bibitem{HS3}
\leavevmode\vrule height 2pt depth -1.6pt width 23pt, {\em  Piecewise dominant sequences and the cocenter of the
cyclotomic quiver Hecke algebras}, Math. Zeit., {\bf 303}(90) (2023), https://doi.org/10.1007/s00209-023-03251-4.

\bibitem{HS4}
\leavevmode\vrule height 2pt depth -1.6pt width 23pt, {\em Proof of the Center Conjectures for the cyclotomic Hecke and KLR algebras of type $A$},
preprint, arXiv:2211.07069v3, (2023).

\bibitem{KK}
{\sc S.~J. Kang and M.~Kashiwara}, {\em Categorification of highest weight modules via Khovanov-Lauda-Rouquier algebras}, Invent. Math., {\bf 190} (2012), 699--742.


\bibitem{KL1}
{\sc M.~Khovanov and A.D.~Lauda}, {\em A diagrammatic approach to categorification of quantum groups, I}, Represent. Theory, {\bf 13} (2009), 309--347.

\bibitem{KL2}
\leavevmode\vrule height 2pt depth -1.6pt width 23pt,  {\em A diagrammatic approach to categorification of quantum groups, II}, Trans. Amer. Math. Soc., {\bf 363} (2011), 2685--2700.



\bibitem{MT1}
{\sc A.~Mathas and D.~Tubenhauer}, {\em Cellularity and subdivision of KLR and weighted KLRW algebras}, Math. Ann.,
https://doi.org/10.1007/s00208-023-02660-4.

\bibitem{MT2}
\leavevmode\vrule height 2pt depth -1.6pt width 23pt, {\em Cellularity for weighted KLRW algebras of types $B, A^{(2)}, D^{(2)}$}, J. Lond. Math. Soc., {\bf 107}(2) (2023), 1002--1044.


\bibitem{RW}
{\sc S.~Riche and G.~Williamson}, {\em Tilting modules and the $p$-canonical bases}, Asterisque, {\bf 397}, 2018.

\bibitem{Rou1}
{\sc R.~Rouquier}, {\em $2$-Kac--Moody algebras}, preprint, math.RT/0812.5023v1, 2008.

\bibitem{Rou2}
\leavevmode\vrule height 2pt depth -1.6pt width 23pt, {\em Quiver Hecke algebras and 2-Lie algebras}, Algebr. Colloq. {\bf 19} (2012), 359--410.


\bibitem{SVV}
{\sc P.~Shan, M.~Varagnolo and E.~Vasserot}, {\em On the center of quiver-Hecke algebras}, Duke Math. J., {\bf 166}(6) (2017), 1005--1101.

\bibitem{VV}
{\sc M.~Varagnolo and E.~Vasserot}, {\em Canonical bases and KLR algebras}, J. reine angew. Math., {\bf 659} (2011), 67--100.

\bibitem{W0}
{\sc B.~Webster}, {\em Center of KLR algebras and cohomology rings of quiver varieties}, preprint, math.RT/1504.04401v2, 2015.

\bibitem{W1}
\leavevmode\vrule height 2pt depth -1.6pt width 23pt,  {\em Rouquier's conjecture and diagrammatic algebra}, Forum Math. Sigma, {\bf 5} (2017), e27, 71 pages.

\bibitem{W2}
\leavevmode\vrule height 2pt depth -1.6pt width 23pt,  {\em Weighted Khovanov-Lauda-Rouquier algebras}, Doc. Math., {\bf 24} (2019), 209¨C--250.

\end{thebibliography}
\end{document}